\documentclass[11pt]{article}
\usepackage[top=1in, bottom=1in, left=1in, right=1in]{geometry}

\usepackage[linktocpage,colorlinks,linkcolor=blue,anchorcolor=blue,citecolor=blue,urlcolor=blue,pagebackref]{hyperref}
\usepackage{euscript,mdframed}
\usepackage{microtype,todonotes,relsize}
\usepackage{amsmath,amsthm}
\usepackage{algpseudocode}

\usepackage{accents}
\usepackage{comment,url,graphicx,relsize}
\usepackage{amssymb,amsfonts,amsmath,amsthm,amscd,dsfont,mathrsfs,mathtools,nicefrac, bm}
\usepackage{float,psfrag,epsfig,color,xcolor,url}
\usepackage{epstopdf,bbm,mathtools,enumitem}
\usepackage{subfigure}
\usepackage[ruled,vlined]{algorithm2e}
\usepackage{tablefootnote}
\usepackage{diagbox}
\usepackage{tikz}
\usetikzlibrary{calc,patterns,angles,quotes}
\usepackage{makecell}
\usepackage{multirow}
\usepackage{booktabs}
\usepackage[authoryear,round]{natbib}
\newtheorem{theorem}{Theorem}[section]
\newtheorem{lemma}{Lemma}[section]
\newtheorem{corollary}{Corollary}[section]

\newtheorem{assumption}{Assumption}[section]
\newtheorem{remark}{Remark}[section]

\newcommand{\argmin}{\operatorname{argmin}}

\newcommand{\unif}{\operatorname{Unif}}
\newcommand{\reals}{{\mathbb{R}}}

\newcommand{\inner}[2]{\langle{#1},{#2}\rangle}
\newcommand{\sign}{\mathtt{sign}}
\newcommand{\mE}{\mathbb{E}}
\newcommand{\cP}{{\mathcal{P}}}
\newcommand{\cO}{{\mathcal{O}}}
\newcommand{\cF}{{\mathcal{F}}}
\newcommand{\bx}{{\boldsymbol{x}}}
\newcommand{\bg}{{\boldsymbol{g}}}
\newcommand{\by}{{\boldsymbol{y}}}
\newcommand{\bz}{{\boldsymbol{z}}}
\newcommand{\bu}{{\boldsymbol{u}}}
\newcommand{\bv}{{\boldsymbol{v}}}
\newcommand{\bs}{{\boldsymbol{s}}}
\newcommand{\bc}{{\boldsymbol{c}}}

\title{Non-Stationary Bandit Convex Optimization: A Near-Optimal Algorithm with Two-Point Feedback}
\author{
{Chang He} \thanks{School of Information Management and Engineering, Shanghai University of Finance and Economics. \texttt{ischanghe@gmail.com}}
\and
{Bo Jiang} \thanks{School of Information Management and Engineering, Shanghai University of Finance and Economics. \texttt{isyebojiang@gmail.com}}
\and
{Shuzhong Zhang} \thanks{Department of Industrial and System Engineering, University of Minnesota. \texttt{zhangs@umn.edu}}
}

\begin{document}
\maketitle
\begin{abstract}
    This paper studies bandit convex optimization in non-stationary environments with two-point feedback, using dynamic regret as the performance measure. We propose an algorithm based on bandit mirror descent that extends naturally to non-Euclidean settings. Let $T$ be the total number of iterations and $\mathcal{P}_{T,p}$ the path variation with respect to the $\ell_p$-norm. In Euclidean space, our algorithm matches the optimal regret bound $\mathcal{O}(\sqrt{dT(1+\mathcal{P}_{T,2})})$, improving upon \citet{zhao2021bandit} by a factor of $\mathcal{O}(\sqrt{d})$. Beyond Euclidean settings, our algorithm achieves an upper bound of $\mathcal{O}(\sqrt{d\log(d)T\log(T)(1 + \mathcal{P}_{T,1})})$ on the simplex, which is nearly optimal up to log factors. For the cross-polytope, the bound reduces to $\mathcal{O}(\sqrt{d\log(d)T(1+\mathcal{P}_{T,p})})$ for some $p = 1 + 1/\log(d)$.  
\end{abstract}
\section{Introduction}
We study the problem of Bandit Convex Optimization (BCO) in non-stationary environments with two-point feedback \citep{zhao2021bandit}, which has successfully modeled many real-world scenarios where the feedback is incomplete \citep{hazan2009efficient}. For example, emerging online network tasks such as fog computing in the Internet of Things require online decisions to flexibly adapt to changing user preferences \citep{chen2018bandit}. This problem can be viewed as a repeated game between a learner and an adversary as follows: At each iteration $t$, the adversary selects a convex loss function $f_t$ on $\mathbb{R}^d$, which remains unknown to the learner. The learner then chooses two points, $\bx^{t,+}$ and $\bx^{t,-}$, from a compact convex set $\mathcal{X} \subseteq \mathbb{R}^d$ and suffers the corresponding losses, $f_t(\bx^{t,+})$ and $f_t(\bx^{t,-})$. In non-stationary environments, the learner's goal is to minimize the \textit{dynamic regret} introduced by \cite{zinkevich2003online}, defined as
\begin{equation*}
   \sum_{t=1}^T \frac{f_t\left(\bx^{t,+}\right) + f_t\left(\bx^{t,-}\right)}{2}- \sum_{t=1}^T f_t(\bu^t).
\end{equation*}
This quantity measures the difference between the learner's cumulative loss and that of a comparator sequence $\bu^1, \ldots, \bu^T \in \mathcal{X}$. The comparator sequence is not known to the learner; our goal is therefore to obtain bounds that hold universally over all comparator sequences. In particular, if the comparator sequence is chosen as $\bu^1 = \ldots = \bu^T = \argmin_{\bx \in \mathcal{X}} \ \sum_{t=1}^T f_t(\bx)$, then the dynamic regret reduces to the well-studied static regret \citep{agarwal2010optimal,duchi2015optimal,shamir2017optimal}, and the dynamic regret bound automatically adapts to stationary environments. In a setting parallel to ours, \citet{liu2025non} recently provided a comprehensive dynamic regret analysis under one-point feedback.

Dynamic regret analysis has attracted considerable attention due to its flexibility, as static regret can be too optimistic and may not hold in non-stationary environments, where data are evolving and the optimal decision is drifting over time \citep{zinkevich2003online,besbes2015non,zhang2018adaptive,zhao2024adaptivity}. In the setting where full information about the loss function is available, \citet{zinkevich2003online} proved that online gradient descent achieves a regret bound of $\mathcal{O}(\sqrt{T}(1 + \mathcal{P}_{T, 2}))$, where $\mathcal{P}_{T, 2}$ represents the path variation of the comparator sequence $\bu^1, \ldots, \bu^T$ with respect to the $\ell_2$-norm, defined as $\mathcal{P}_{T, 2} \triangleq \sum_{t=1}^{T-1} \|\bu^{t+1} - \bu^t\|_2$. \cite{zhang2018adaptive} showed that this upper bound is not tight by establishing a lower bound of $\Omega(\sqrt{T(1 + \mathcal{P}_{T, 2})})$, and proposed an online algorithm that achieves the optimal result. Moreover, many studies have focused on a special case that the comparator sequence is chosen as $\bu^t = \argmin_{\bx \in \mathcal{X}} f_t(\bx)$, $t = 1, \ldots, T$; see, for example, \citet{jadbabaie2015online, yang2016tracking, mokhtari2016online}. This choice of $\{\bu^t\}_{t=1}^T$ corresponds to the worst-case dynamic regret, which, however, is not the interest of our paper.

While partial information about the loss function is available, research on designing algorithms to minimize the dynamic regret in BCO remains relatively unexplored. \citet{yang2016tracking, chen2018bandit} provided worst-case dynamic regret analyses; however, their algorithms require certain parameters as input, such as the budget of the path variation, which are generally unknown in advance. \citet{zhao2021bandit} took the first step towards developing algorithms that achieve universal dynamic regret guarantees for BCO problems. Based on the idea of online ensemble \citep{zhou2012ensemble}, they proposed a parameter-free bandit gradient descent algorithm\footnote{``Parameter-free" means that the algorithm does not require prior knowledge of the comparator sequence.} and established an upper bound of $\mathcal{O}(d\sqrt{T(1 + \mathcal{P}_{T, 2})})$ with two-point feedback. However, the corresponding lower bound is $\Omega(\sqrt{d T(1 + \mathcal{P}_{T, 2})})$, revealing a clear gap in terms of the dimension $d$. Furthermore, existing studies on non-stationary BCO rely on the Euclidean structure, making it difficult to extend these algorithms to non-Euclidean settings, such as those where decisions are naturally structured on the simplex or the cross-polytope \citep{duchi2010composite,shao2024improved}. 

Motivated by the optimal algorithm designed for static regret \citep{shamir2017optimal}, we propose an algorithm with two-point feedback to address the above issues. Our algorithm is built on bandit mirror descent, allowing the use of the $\ell_p$-norm path variation $\mathcal{P}_{T,p} \triangleq \sum_{t=1}^{T-1} \|\bu^{t+1} - \bu^t\|_p$ to quantify the dynamic regret, thereby extending beyond the standard Euclidean structure. Like previous algorithms with two-point feedback, our algorithm is based on a random gradient estimator. Given a function $f$ and a point $\bx$, the estimator queries $f$ at two random locations near $\bx$ and computes a random vector, denoted as $\bg$, whose expectation is the gradient of a smoothed version of $f$. The key idea for improving dimensional dependence is to incorporate the second-order moment $\mE[\|\bg\|^2_{p^*}]$ into the regret analysis, as $\mE[\|\bg\|^2_{p^*}]$ exhibits significantly lower dependence on the dimension compared to its uniform upper bound $\|\bg\|_{p^*}^2$. Based on this insight, we propose a parameter-free bandit mirror descent algorithm which integrates bandit mirror descent with an expert aggregation over a grid of learning rates. This algorithm eliminates the need for prior knowledge of the comparator sequence and automatically adapts to non-Euclidean settings. In Euclidean space, the dynamic regret bound achieved by our algorithm is optimal and matches the lower bound established in \cite{zhao2021bandit}. On the simplex, the bound is near-optimal, differing only by log factors compared to the lower bound established in Theorem \ref{thm:simplex-dynamic-lower}. A summary of the results across different setups is provided in Table \ref{table.summary}.
\begin{table*}[h]
\centering
\begin{tabular}{@{}ccc@{}}
\toprule
Setup           & Upper bound & Lower bound \\ \midrule
\multirow{2}{*}{Euclidean}      
               & $\mathcal{O}\left(\sqrt{dT\left(1 + \mathcal{P}_{T, 2}\right)}\right)$  & $\Omega\left(\sqrt{dT\mathcal{P}_{T, 2}}\right)$  \\[0.2cm]
               & Corollary \ref{proposition.BMD 2 norm}  & Theorem 5 in \cite{zhao2021bandit} \\ \midrule
\multirow{2}{*}{Simplex}        
               & $\mathcal{O}\left(G \sqrt{d\log(d)T\log(T)\left(1 + \mathcal{P}_{T,1}\right)}\right)$ &  $\Omega\left(\sqrt{dT(1 + \mathcal{P}_{T, 1})/\log(d)}\right)$ \\[0.2cm]  
               & Theorem \ref{proposition.BMD simplex with entropy} & Theorem \ref{thm:simplex-dynamic-lower} \\ \midrule
\multirow{2}{*}{Cross-polytope} 
               & $\mathcal{O}\left(\sqrt{d\log(d)T\left(1 + \mathcal{P}_{T,p}\right)}\right)^\dagger$ & \multirow{2}{*}{N/A}  \\[0.2cm]  
               & Corollary \ref{proposition.BMD cross polytope} \\  \bottomrule
\end{tabular} \\ \vspace{0.3cm}
\caption{Summary of dynamic regret bounds for our algorithm under different setups. The result marked with ``$\dagger$" corresponds to $p = 1 + 1/\log(d)$.}
\label{table.summary}
\end{table*}

\paragraph{Notation.} In the paper, we denote vectors by boldface letters (e.g., $\bx, \by$), and use non-bold letters with subscripts to represent their coordinates (e.g., $x_j, y_j$). The notations $\cO$, $\Omega$, and $\Theta$ follow the standard asymptotic conventions. We assume that $p, q \in[1, \infty], d \geq 3$, and set $p^*, q^* \in[1, \infty]$ such that $1/p + 1/p^*=1$ and $1/q + 1/q^*=1$, with the usual convention $1 / \infty=0$, $0 \log(0) = 0$. The filtration generated by the random vectors sampled upon iteration $t$ is denoted by $\cF_t$, i.e., $\cF_0=\{\emptyset, \Omega\}$ and $\cF_t \triangleq \sigma\left(\bs^{k} \mid k=1, \ldots, t\right), \forall t \geq 1$, where $\sigma$ denotes the $\sigma$-algebra generated by the random vectors. We let $\langle\cdot, \cdot\rangle$ be the standard inner product in $\mathbb{R}^d$. For any $p \in[1, \infty]$, we introduce the unit $\ell_p$-ball and $\ell_p$-sphere respectively as $ \mathbb{B}_p^d \triangleq\left\{\bx \in \mathbb{R}^d:\|\bx\|_p \le 1\right\}$ and $\partial \mathbb{B}_p^d \triangleq\left\{\bx \in \mathbb{R}^d:\|\bx\|_p=1\right\}$. We denote by $\bx \mapsto \sign(\bx)$ the component-wise sign function (defined at $0$ as $1$). Without loss of generality, we assume that $\mathcal{X}$ has a nonempty interior, denoted by $\operatorname{int}(\mathcal{X})$. Given a divergence generating function $\psi : \operatorname{int}(\mathcal{X}) \rightarrow \reals$, the Bregman divergence induced by $\psi$ is defined as $B_{\psi}(\bx;\by) \triangleq \psi(\bx) - \psi(\by) - \inner{\nabla \psi(\by)}{\bx - \by}$. 

\section{Building Block: Bandit Mirror Descent}\label{section.bmd}

\paragraph{Exploration.} Let $\{\by^t\}$ denote the sequence of iterates generated by Algorithm \ref{alg.BMD}. Our exploration strategy employs the following gradient estimator with two-point feedback (Line \ref{line:exploration} in Algorithm \ref{alg.BMD}), based on the $\ell_1$-sphere smoothing technique proposed by \citet{akhavan2022gradient}:  
\begin{equation}\label{eq.gradient estimator}
    \bg^t = \frac{d}{2\mu}\left(f_t(\by^t + \mu \bs^t) - f_t(\by^t - \mu \bs^t)\right) \sign(\bs^t),
\end{equation}
where $\bs^t \sim \unif(\partial \mathbb{B}_1^d)$ follows a uniform distribution over the unit $\ell_1$-sphere, and $\mu > 0$ is the smoothing parameter. The intuition behind this estimator comes from Stokes' theorem, which states that under certain regularity conditions,
$$
\int_D \nabla f(\bx) \mathrm{d} \bx=\int_{\partial D} f(\bx) \overset{\rightarrow}{n}(\bx) \ \mathrm{d} S(\bx),
$$
where $\partial D$ denotes the boundary of $D$, $\overset{\rightarrow}{n}(\bx)$ is the outward normal vector to $\partial D$, and $\mathrm{d} S(\bx)$ is the surface measure. When $D =  \mathbb{B}_1^d$, we have $\overset{\rightarrow}{n}(\bx) = \sign(\bx)/\sqrt{d}$, leading to the derivation of the gradient estimator \eqref{eq.gradient estimator}. For a more detailed discussion, we refer the reader to Section 2 in \citet{akhavan2022gradient}. Since a small perturbation is required to construct the estimator, we impose the following widely used assumption in BCO \citep{flaxman2005online, agarwal2010optimal, zhao2021bandit} to ensure the feasibility of the perturbed points $\by^t \pm \mu \bs^t$.
\begin{assumption}[Bounded Region]\label{assumption.bounded region}
   The feasible set $\mathcal{X}$ contains a ball of radius $r$ centered at the origin and is contained in a ball of radius $R$, namely, $r \mathbb{B}_p^d \subseteq \mathcal{X} \subseteq R \mathbb{B}_p^d$.
\end{assumption}
Then the feasibility issue can be tackled by operating in a slightly smaller set $\mathcal{X}_{\alpha} \triangleq (1-\alpha)\mathcal{X} = \{\by \in \reals^d: \by = (1-\alpha)\bx, \bx \in \mathcal{X}\}$, where $\alpha \in (0, 1)$ is called the shrinkage parameter. It is easy to verify that $\mathcal{X}_{\alpha}$ is compact and convex. Throughout the paper, we use the shorthand $F_\psi^\alpha\triangleq\sup_{\bx\in\mathcal{X}_\alpha}\psi(\bx)-\inf_{\bx\in\mathcal{X}_\alpha}\psi(\bx)$ and $G_\psi^\alpha\triangleq\sup_{\by\in\mathcal{X}_\alpha}\|\nabla\psi(\by)\|_{p^*}$. A suitable choice of $\alpha$ and the smoothing parameter $\mu$ guarantee that the perturbed points $\by^t \pm \mu \bs^t$ remain inside $\mathcal{X}$. The following lemma formalizes the relationship between $\alpha$ and $\mu$ mathematically.

\begin{lemma}\label{lemma.relation between mu and alpha}
   Suppose Assumption \ref{assumption.bounded region} holds and $\alpha \in (0,1)$. If the smoothing parameter and shrinkage parameter satisfy $\mu \le \alpha r$, then for any feasible point $\bx \in (1-\alpha)\mathcal{X}$, we have $\bx + \mu \mathbb{B}_1^d \subseteq \mathcal{X}$. In particular, $\bx \pm \mu \bs^t \in \mathcal{X}$ for every $\bs^t \in \partial \mathbb{B}_1^d$.
\end{lemma}
\begin{proof}
    Since $\|\bz\|_p \le \|\bz\|_1$ for every $\bz \in \reals^d$ and $p \ge 1$, we have $\mathbb{B}_1^d \subseteq \mathbb{B}_p^d$. Therefore, it holds that
    \begin{align*}
        (1-\alpha)\mathcal{X} + \mu \mathbb{B}_1^d
        \subseteq (1-\alpha)\mathcal{X} + \mu \mathbb{B}_p^d
        \subseteq (1-\alpha)\mathcal{X} + \alpha r \mathbb{B}_p^d.
    \end{align*}
    By Assumption \ref{assumption.bounded region}, we have $r \mathbb{B}_p^d \subseteq \mathcal{X}$, which implies that $\alpha r \mathbb{B}_p^d \subseteq \alpha \mathcal{X}$. By convexity of $\mathcal{X}$, we have $ (1-\alpha)\mathcal{X} + \alpha \mathcal{X} \subseteq \mathcal{X}$. Combining the above inclusions yields $(1-\alpha)\mathcal{X} + \mu \mathbb{B}_1^d \subseteq \mathcal{X}$, which proves the claim.
\end{proof}

In the sequel, unless stated otherwise, we instantiate the shrinkage parameter as $\alpha = \mu / r$, which is the largest admissible choice allowed by Lemma \ref{lemma.relation between mu and alpha}. We now proceed to present the properties of the gradient estimator \eqref{eq.gradient estimator}, under the assumption that the loss functions are Lipschitz continuous with respect to the $\ell_q$-norm. This assumption generalizes the setting in \cite{zhao2021bandit} as a special case.

\begin{assumption}[$\ell_q$-norm Lipschitz Continuity]\label{assumption.G Lipschitz}
   All the functions are $G$-Lipschitz continuous with respect to $\ell_q$-norm over the feasible set $\mathcal{X}$, that is, for all $\bx, \by \in \mathcal{X}$, we have $\left|f_t(\bx)-f_t(\by)\right| \leq G\|\bx-\by\|_q, \ \forall t = 1, \ldots, T$.
\end{assumption}
\begin{lemma}[Lemmas 2 and 4 in \cite{akhavan2022gradient}]\label{lemma.properties of gradient estimator}
   Suppose Assumptions \ref{assumption.bounded region} and \ref{assumption.G Lipschitz} hold. Let $f^\mu_t (\by) \triangleq \mE_{\bs \sim \unif(\mathbb{B}_1^d)} \left[f_t(\by + \mu \bs)\right]$ be the smoothing function. Then for all $t=1, \ldots, T$, it holds that
   \begin{enumerate}
       \item $f^\mu_t(\cdot)$ is convex and satisfies $|f^\mu_t(\by)-f_t(\by)|\le \zeta_q(d)G\mu$ for every $\by$ such that $\by+\mu\mathbb{B}_1^d\subseteq\mathcal{X}$, where the constant $\zeta_q(d)$ is defined as
        \begin{align*}
           \zeta_q(d) \triangleq \begin{cases}
              qd^{\frac{1}{q}} / (d+1), \ &q \in [1,\log(d)), \\ 
               e \log(d) / (d+1), \ &q \ge \log(d);
           \end{cases}
        \end{align*} 
       \item  $f^\mu_t(\cdot)$ is differentiable, and satisfies $\nabla f^\mu_t(\by^t) = \mE_{\bs^t \sim \unif(\partial \mathbb{B}_1^d)} \left[\bg^t \mid \cF_{t-1} \right]$, where $\bg^t$ is defined in Equation \eqref{eq.gradient estimator};
       \item the second-order moment of gradient estimator \eqref{eq.gradient estimator} satisfies
        \begin{equation}\label{eq.varaince}
            \mE \left[\|\bg^t\|_{p^*}^2 \mid \cF_{t-1}\right] \le 12(1+\sqrt{2})^2G^2 \xi_{p,q}(d), \ \xi_{p,q}(d) \triangleq d^{1 + \frac{2}{\min\{q,2\}}- \frac{2}{p}}.
        \end{equation}
   \end{enumerate}
\end{lemma}
\begin{remark}
    We do not adopt the widely used $\ell_2$-sphere smoothing technique to construct the gradient estimator \citep{agarwal2010optimal,ghadimi2013stochastic,duchi2015optimal,shamir2017optimal,gao2018information}, as its second-order moment depends on $\ell_2$-norm Lipschitz continuity. When considering non-Euclidean structures, transforming $\ell_q$-norm Lipschitz continuity to the Euclidean norm requires using norm equivalences (see Lemma \ref{lemma.vector norm}), which introduces a large dependence on the dimension.
\end{remark}

\paragraph{Exploitation.} The exploitation strategy follows the Bandit Mirror Descent (\texttt{BMD}) (Line \ref{line:exploitation} in Algorithm \ref{alg.BMD}); the next iterate $\by^{t+1}$ is obtained by seeking the minimizer within the shrunk set $\mathcal{X}_{\alpha}$.
This shrinking step ensures that the sampling procedure at iteration $t+1$ is legal, provided that the choices of $\mu$ and $\alpha$ satisfy the condition specified in Lemma \ref{lemma.relation between mu and alpha}.

\begin{algorithm}[h]
  \caption{Bandit Mirror Descent (\texttt{BMD})}
  \label{alg.BMD}
  \begin{algorithmic}[1]
    \item \textbf{Input:} smoothing parameter $\mu>0$, shrinkage parameter $\alpha>0$, step size $\eta>0$, divergence-generating function $\psi$
    \item \textbf{Initialize:} pick $\by^{1}\in(1-\alpha)\mathcal{X}$
    \item \textbf{For} $t = 1,2,\ldots,T$ \textbf{do}
    \item \quad Generate $\bs^t \sim \unif(\partial \mathbb{B}_1^d)$, and play decisions
            \begin{align*}
                \bx^{t,+} = \by^t + \mu \bs^t, \quad \bx^{t,-} = \by^t - \mu \bs^t;
            \end{align*}
    \item \quad Receive $f_t(\bx^{t,+})$, $f_t(\bx^{t,-})$ as feedback and construct the estimator:\label{line:exploration}
            \begin{align*}
                \bg^t = \frac{d}{2\mu}\left(f_t(\bx^{t,+}) - f_t(\bx^{t,-})\right) \cdot \sign(\bs^t);
            \end{align*}
    \item \quad Update the iterate: \label{line:exploitation}
            \begin{align*}
                \by^{t+1} = \argmin_{\by \in (1-\alpha)\mathcal{X}} \inner{\bg^t}{\by} + \frac{1}{\eta} B_{\psi}(\by;\by^t);
            \end{align*}
    \item \textbf{end For}
  \end{algorithmic}
\end{algorithm}

Before delving into the dynamic regret analysis, we provide a preliminary investigation of the expected dynamic regret. This helps identify the crucial term related to \texttt{BMD} and distinguish it from trivial terms, which is a common approach in regret analysis \citep{saha2011improved,dekel2015bandit,zhao2021bandit}. Let $\bv^t \triangleq (1-\alpha)\bu^t \in \mathcal{X}_{\alpha}$, $t = 1, \ldots, T$ be the scaled comparator sequence and take expectation with respect to all randomness $\bs^1, \ldots, \bs^T$. We decompose the expected dynamic regret of loss functions into three terms:
\begin{equation}\label{eq.three terms decomposition}
\begin{aligned}
 &\mE \left[\sum_{t=1}^T \frac{f_t(\bx^{t,+}) + f_t(\bx^{t,-})}{2} - f_t(\bu^t)\right] \\
   \le \  &\underbrace{\mE\left[\sum_{t=1}^T f_t(\by^t) - f_t(\bv^t)\right]}_{\texttt{term (a)}} + \underbrace{\mE\left[\sum_{t=1}^T \frac{f_t(\bx^{t,+}) + f_t(\bx^{t,-})}{2} - f_t(\by^t)\right]}_{\texttt{term (b)}} + \underbrace{\mE \left[\sum_{t=1}^T f_t(\bv^t) - f_t(\bu^t)\right]}_{\texttt{term (c)}}.
\end{aligned}
\end{equation}
Since loss functions are $G$-Lipschitz continuous with respect to $\ell_q$-norm over domain $\mathcal{X}$, then both \texttt{term (b)} and \texttt{term (c)} can be upper bounded as follows.
\begin{lemma}\label{lemma.bound term b and c}
    Suppose Assumptions \ref{assumption.bounded region} and \ref{assumption.G Lipschitz} hold, and let $\alpha = \mu/r$. It holds that
    \begin{align*}
           \texttt{term (b)} &\le G T\mu, \\
           \texttt{term (c)} &\le \frac{G R \upsilon_{p,q}(d) T\mu}{r},
    \end{align*}
    where the constant $\upsilon_{p,q}(d) \triangleq d^{1/q - 1/\max\{q,p\}}$.
\end{lemma}
\begin{proof}
    Due to the Lipschitz continuity and $\bx^{t,\pm} = \by^t \pm \mu \bs^t$, then for all $t = 1, \ldots, T$, we have
    \begin{align*}
        \left|\frac{f_t(\bx^{t,+}) + f_t(\bx^{t,-})}{2} - f_t(\by^t)\right|
        \le \frac{G}{2}\big(\|\bx^{t,+} - \by^t\|_q + \|\bx^{t,-} - \by^t\|_q\big)
        = G\mu \|\bs^t\|_q
        \le G\mu \|\bs^t\|_1
        = G\mu,
    \end{align*}
    where the last equality holds because $\bs^t \in \partial \mathbb{B}_1^d$. Summing over $t$ immediately gives $\texttt{term (b)} \le G T\mu$. Similarly, using the definition $\bv^t = (1-\alpha)\bu^t$ yields
    \begin{align*}
        \left|f_t(\bv^t) - f_t(\bu^t)\right|
        \le G\|\bv^t - \bu^t\|_q
        = G\alpha \|\bu^t\|_q
        = \frac{G\mu}{r}\|\bu^t\|_q.
    \end{align*}
    By Lemma \ref{lemma.vector norm}, we obtain
    \begin{align*}
        \|\bu^t\|_q \le d^{1/q - 1/\max\{q,p\}}\|\bu^t\|_p = \upsilon_{p,q}(d)\|\bu^t\|_p.
    \end{align*}
    Since $\mathcal{X} \subseteq R \mathbb{B}_p^d$, we have $\|\bu^t\|_p \le R$ and therefore
    \begin{align*}
        \left|f_t(\bv^t) - f_t(\bu^t)\right|
        \le \frac{G R \upsilon_{p,q}(d)\mu}{r}.
    \end{align*}
    Summing over $t$ proves $\texttt{term (c)} \le G R \upsilon_{p,q}(d)T\mu/r$.
\end{proof}

Regarding \texttt{term (a)}, we take the smoothing functions $f^\mu_t$, $t = 1, \ldots, T$ defined in Lemma \ref{lemma.properties of gradient estimator} as surrogates. This yields the following bound
\begin{equation}\label{eq.bound term a}
    \texttt{term (a)} \le \underbrace{\mE\left[\sum_{t=1}^T f^\mu_t(\by^t) - f^\mu_t(\bv^t)\right]}_{\texttt{term (d)}} + 2G \zeta_{q}(d) T\mu.
\end{equation}
Consequently, the analysis of the expected dynamic regret for \texttt{BMD} reduces to bounding the crucial \texttt{term (d)}. To achieve this, the following lemma provides a fundamental analysis of the mirror descent update.

\begin{lemma}\label{lemma.pathwise mirror descent}
    Suppose Assumption \ref{assumption.bounded region} holds. Let $\psi: \operatorname{int}(\mathcal{X}) \rightarrow \reals$ be a divergence-generating function that is $\lambda$-strongly convex with respect to the $\ell_p$-norm ($\lambda > 0$). Let $\{\by^t\}_{t=1}^{T+1}$ be generated by the mirror descent update in Algorithm \ref{alg.BMD} with an arbitrary sequence of vectors $\bg^1,\ldots,\bg^T \in \reals^d$. Then for every comparator sequence $\bv^1,\ldots,\bv^T \in \mathcal{X}_{\alpha}$,
    \begin{equation*}
        \sum_{t=1}^T \inner{\bg^t}{\by^t - \bv^t}
        \le \frac{F_\psi^\alpha + B_{\psi}(\bv^1;\by^1)}{\eta}
        + \frac{G_\psi^\alpha}{\eta}\sum_{t=1}^{T-1}\|\bv^{t+1} - \bv^t\|_p
        + \frac{\eta}{2\lambda}\sum_{t=1}^T \|\bg^t\|_{p^*}^2.
    \end{equation*}
\end{lemma}
\begin{proof}
    Since $0 \in \operatorname{int}(\mathcal{X})$ and $\alpha \in (0,1)$, convexity implies $(1-\alpha)\mathcal{X} \subseteq \operatorname{int}(\mathcal{X})$. Hence every iterate $\by^t \in \mathcal{X}_{\alpha}$ lies in the domain where $\nabla \psi$ is well defined. Fix an arbitrary realization of $\bg^1,\ldots,\bg^T$. From the optimality condition for the mirror descent update (Line \ref{line:exploitation} in Algorithm \ref{alg.BMD}), we have
    \begin{align*}
        \inner{\eta \bg^t + \nabla \psi(\by^{t+1}) - \nabla \psi(\by^t)}{\bv - \by^{t+1}} \ge 0,
        \quad \forall \bv \in \mathcal{X}_{\alpha}.
    \end{align*}
    Taking $\bv = \bv^t$ and rearranging gives
    \begin{align*}
        \eta \inner{\bg^t}{\by^t - \bv^t}
        \le \inner{\nabla \psi(\by^{t+1}) - \nabla \psi(\by^t)}{\bv^t - \by^{t+1}}
        + \eta \inner{\bg^t}{\by^t - \by^{t+1}}.
    \end{align*}
    By the three-point identity (take $\bx = \by^{t+1}$, $\by = \by^t$, and $\bz = \bv^t$ in Lemma \ref{lemma.three point bregman}), we have
    \begin{align*}
        \eta \inner{\bg^t}{\by^t - \bv^t}
        \le \ &B_{\psi}(\bv^t;\by^t) - B_{\psi}(\bv^t;\by^{t+1}) - B_{\psi}(\by^{t+1};\by^t) + \eta \inner{\bg^t}{\by^t - \by^{t+1}} \\
        \le \ &B_{\psi}(\bv^t;\by^t) - B_{\psi}(\bv^t;\by^{t+1}) - \frac{\lambda}{2}\|\by^{t+1} - \by^t\|_p^2 + \eta \|\bg^t\|_{p^*}\|\by^t - \by^{t+1}\|_p \\
        \le \ &B_{\psi}(\bv^t;\by^t) - B_{\psi}(\bv^t;\by^{t+1}) + \frac{\eta^2}{2\lambda}\|\bg^t\|_{p^*}^2,
    \end{align*}
    where the second inequality uses $\lambda$-strong convexity of $\psi$, and the last inequality is Young's inequality. Summing over $t=1,\ldots,T$ yields
    \begin{equation}\label{eq:middle step of MD}
        \sum_{t=1}^T \inner{\bg^t}{\by^t - \bv^t}
        \le \frac{1}{\eta}\sum_{t=1}^T \big(B_{\psi}(\bv^t;\by^t) - B_{\psi}(\bv^t;\by^{t+1})\big)
        + \frac{\eta}{2\lambda}\sum_{t=1}^T \|\bg^t\|_{p^*}^2.
    \end{equation}
    The Bregman terms telescope as
    \begin{align*}
        \sum_{t=1}^T \big(B_{\psi}(\bv^t;\by^t) - B_{\psi}(\bv^t;\by^{t+1})\big)
        = \ &B_{\psi}(\bv^1;\by^1) - B_{\psi}(\bv^T;\by^{T+1}) + \sum_{t=1}^{T-1}\big(B_{\psi}(\bv^{t+1};\by^{t+1}) - B_{\psi}(\bv^t;\by^{t+1})\big).
    \end{align*}
    For each $t=1,\ldots,T-1$, the definition of Bregman divergence gives
    \begin{align*}
        B_{\psi}(\bv^{t+1};\by^{t+1}) - B_{\psi}(\bv^t;\by^{t+1})
        = \psi(\bv^{t+1}) - \psi(\bv^t) - \inner{\nabla \psi(\by^{t+1})}{\bv^{t+1} - \bv^t}.
    \end{align*}
    Therefore, we have
    \begin{align*}
        B_{\psi}(\bv^{t+1};\by^{t+1}) - B_{\psi}(\bv^t;\by^{t+1})
        \le \psi(\bv^{t+1}) - \psi(\bv^t) + G_\psi^\alpha\|\bv^{t+1} - \bv^t\|_p.
    \end{align*}
    Since $B_{\psi}(\bv^T;\by^{T+1}) \ge 0$ and $\psi(\bv^T) - \psi(\bv^1) \le F_\psi^\alpha$, we obtain
    \begin{align*}
        \sum_{t=1}^T \big(B_{\psi}(\bv^t;\by^t) - B_{\psi}(\bv^t;\by^{t+1})\big)
        \le F_\psi^\alpha + B_{\psi}(\bv^1;\by^1) + G_\psi^\alpha\sum_{t=1}^{T-1}\|\bv^{t+1} - \bv^t\|_p.
    \end{align*}
    Substituting the above inequality into the previous bound \eqref{eq:middle step of MD} proves the result.
\end{proof}

In previous work, \citet{zhao2021bandit} leveraged the fact that the gradient estimator \eqref{eq.gradient estimator}\footnote{\cite{zhao2021bandit} used the $\ell_2$-sphere smoothing technique, but the gradient estimator can be similarly bounded.} is almost surely bounded as follows:
\begin{equation}\label{eq.direct boundness of gradient estimator}
    \begin{aligned}
     \|\bg^t\|_{p^*} = \ &\frac{d}{2\mu}\left|f_t(\by^t + \mu \bs^t) - f_t(\by^t - \mu \bs^t)\right| \cdot \|\sign(\bs^t)\|_{p^*} \\
     \le \ &dG \|\bs^t\|_{q} \cdot \|\sign(\bs^t)\|_{p^*} \\
     = \ &\mathcal{O}(d\|\sign(\bs^t)\|_{p^*}).
    \end{aligned}
\end{equation}
By selecting appropriate surrogates, bandit gradient descent can then be interpreted as a randomized online gradient descent applied to these surrogates. Hence, they referenced existing dynamic regret bounds for the online gradient descent algorithm \citep{zinkevich2003online}. Certainly, we can also follow this approach by using the result of online mirror descent, but it leads to a large dependence on the dimension. For example, when $p = 2$, we have $\|\bg^t\|_2 = \mathcal{O}(d^{3/2})$, whereas the optimal dependence on dimension is $\mathcal{O}(\sqrt{d})$. The key point here is the incorporation of the second-order moment bound \eqref{eq.varaince} of the gradient estimator, which is also the insight behind the optimal bandit algorithm in the static regret setting \citep{shamir2017optimal}. 

\begin{lemma}\label{lemma.bound term d}
    Suppose Assumptions \ref{assumption.bounded region} and \ref{assumption.G Lipschitz} hold. Let $\psi: \operatorname{int}(\mathcal{X}) \rightarrow \reals$ be a divergence-generating function that is $\lambda$-strongly convex with respect to the $\ell_p$-norm ($\lambda > 0$). Then for any step size $\eta > 0$, the \texttt{term (d)} can be upper bounded by
    \begin{equation*}
        \mE\left[\sum_{t=1}^T f^\mu_t(\by^t) - f^\mu_t(\bv^t)\right]
        \le \frac{F_\psi^\alpha + B_{\psi}(\bv^{1};\by^{1})}{\eta} + \frac{G_\psi^\alpha \mathcal{P}_{T,p}}{\eta} + \frac{6(1+\sqrt{2})^2 G^2 \xi_{p,q}(d) T}{\lambda}\eta.
    \end{equation*}
    The expectation here is taken with respect to all randomness $\bs^1, \ldots, \bs^T$.
\end{lemma}
\begin{proof}
    For each $t$, the iterate $\by^t$ is $\cF_{t-1}$-measurable. First note that $\mE[\bg^t \mid \cF_{t-1}] = \nabla f^\mu_t(\by^t)$ due to Lemma \ref{lemma.properties of gradient estimator}. Since $f^\mu_t$ is convex, we have
    \begin{align*}
        f^\mu_t(\by^t) - f^\mu_t(\bv^t)
        \le \inner{\nabla f^\mu_t(\by^t)}{\by^t - \bv^t}
        = \mE\left[\inner{\bg^t}{\by^t - \bv^t}\mid \cF_{t-1}\right].
    \end{align*}
    Taking total expectation and summing over $t$ yields
    \begin{align*}
        \mE\left[\sum_{t=1}^T f^\mu_t(\by^t) - f^\mu_t(\bv^t)\right]
        \le \mE\left[\sum_{t=1}^T \inner{\bg^t}{\by^t - \bv^t}\right].
    \end{align*}
    Now apply Lemma \ref{lemma.pathwise mirror descent} to the realized gradient estimator sequence $\bg^1,\ldots,\bg^T$ and the comparator sequence $\{\bv^t\}_{t=1}^T$. Taking expectation gives
    \begin{align*}
        \mE\left[\sum_{t=1}^T f^\mu_t(\by^t) - f^\mu_t(\bv^t)\right]
        \le \frac{F_\psi^\alpha + B_{\psi}(\bv^1;\by^1)}{\eta}
        + \frac{G_\psi^\alpha}{\eta}\sum_{t=1}^{T-1}\|\bv^{t+1} - \bv^t\|_p
        + \frac{\eta}{2\lambda}\sum_{t=1}^T \mE\|\bg^t\|_{p^*}^2.
    \end{align*}
    Since $\bv^t = (1-\alpha)\bu^t$, we have
    \[
       \sum_{t=1}^{T-1}\|\bv^{t+1} - \bv^t\|_p
       = (1-\alpha)\sum_{t=1}^{T-1}\|\bu^{t+1} - \bu^t\|_p
       \le \mathcal{P}_{T,p}.
    \]
    Moreover, by \eqref{eq.varaince}, it holds that
    \[
       \mE\|\bg^t\|_{p^*}^2
       = \mE\left[\mE\left[\|\bg^t\|_{p^*}^2 \mid \cF_{t-1}\right]\right]
       \le 12(1+\sqrt{2})^2 G^2 \xi_{p,q}(d).
    \]
    Substituting the last two inequalities completes the proof.
\end{proof}

Fix the initialization $\by^1 \in \mathcal{X}_{\alpha}$. We introduce the following constant $A_\psi^\alpha$ satisfying
\begin{equation}\label{eq.definition A psi}
    A_\psi^\alpha \ge F_\psi^\alpha + \sup_{\bv \in \mathcal{X}_{\alpha}} B_{\psi}(\bv;\by^1) + \max\{1, \lambda\}R^2.
\end{equation}
Note that $A_\psi^\alpha$ can be chosen explicitly from the geometry of $\mathcal{X}$ and the regularizer $\psi$. Now let us assume temporarily that the path variation $\mathcal{P}_{T,p}$ is known. In view of Lemma \ref{lemma.bound term d}, the natural step size is
\begin{equation}\label{eq.optimal step size}
     \eta_\star = \sqrt{\frac{A_\psi^\alpha + G_\psi^\alpha \mathcal{P}_{T,p}}{6(1+\sqrt{2})^2 G^2 \xi_{p,q}(d)T / \lambda}}.
\end{equation}
Let $c_{\mu}(p,q,d) \triangleq 1 + 2\zeta_q(d) + \frac{R}{r}\upsilon_{p,q}(d)$. By properly setting the smoothing parameter, the dynamic regret bounds for \texttt{BMD} algorithm can be derived. The results are summarized below, following from the combination of the three-term decomposition \eqref{eq.three terms decomposition} with Lemmas \ref{lemma.bound term b and c} and \ref{lemma.bound term d}.
\begin{theorem}\label{thm.dynamic regret of BMD}
      Suppose Assumptions \ref{assumption.bounded region} and \ref{assumption.G Lipschitz} hold, and the path variation $\mathcal{P}_{T,p}$ is known. Let $\psi: \operatorname{int}(\mathcal{X}) \rightarrow \reals$ be a divergence-generating function that is $\lambda$-strongly convex with respect to the $\ell_p$-norm ($\lambda > 0$). Choose
     \begin{align*}
        \eta &= \eta_\star = \sqrt{\frac{A_\psi^\alpha + G_\psi^\alpha \mathcal{P}_{T,p}}{6(1+\sqrt{2})^2 G^2 \xi_{p,q}(d)T / \lambda}}, \\
        \mu &= \min\left\{\frac{R\sqrt{\xi_{p,q}(d)}}{\sqrt{\lambda T} c_{\mu}(p,q,d)},\ \frac{r}{2}\right\},
     \end{align*}
     and set $\alpha = \mu/r$. Then the expected dynamic regret of \texttt{BMD} satisfies
     \begin{align*}
        \mE \left[\sum_{t=1}^T \frac{f_t(\bx^{t,+}) + f_t(\bx^{t,-})}{2} - f_t(\bu^t)\right]
        = \mathcal{O}\left(G  \sqrt{\left(A_\psi^\alpha + G_\psi^\alpha \mathcal{P}_{T,p}\right)\xi_{p,q}(d) T  / \lambda}\right),
     \end{align*}
     where the expectation is taken with respect to all randomness $\bs^1, \ldots, \bs^T$.
\end{theorem}
\begin{proof}
    Combining \eqref{eq.three terms decomposition}, \eqref{eq.bound term a}, Lemmas \ref{lemma.bound term b and c} and \ref{lemma.bound term d}, and the definition of $A_\psi^\alpha$ in \eqref{eq.definition A psi}, we obtain
    \begin{align*}
        \mE \left[\sum_{t=1}^T \frac{f_t(\bx^{t,+}) + f_t(\bx^{t,-})}{2} - f_t(\bu^t)\right]
        \le \ &\frac{A_\psi^\alpha + G_\psi^\alpha \mathcal{P}_{T,p}}{\eta}
        + \frac{6(1+\sqrt{2})^2 G^2 \xi_{p,q}(d) T}{\lambda}\eta
        + G T\mu c_{\mu}(p,q,d) \\
        = \ & 2\sqrt{6}(1+\sqrt{2})G\sqrt{\left(A_\psi^\alpha + G_\psi^\alpha\mathcal{P}_{T,p}\right)\xi_{p,q}(d)T/\lambda} + G T\mu c_{\mu}(p,q,d),
    \end{align*}
    where the last step substitutes $\eta = \eta_\star$. If the minimum defining $\mu$ is attained at the first value, then we have
    \[
        GT\mu c_{\mu}(p,q,d)
        = GR\sqrt{\xi_{p,q}(d) T/\lambda} = \mathcal{O}\left(G  \sqrt{\left(A_\psi^\alpha + G_\psi^\alpha \mathcal{P}_{T,p}\right)\xi_{p,q}(d) T  / \lambda}\right),
    \]
    where the last step holds since $R^2 \le A_\psi^\alpha \le A_\psi^\alpha + G_\psi^\alpha \mathcal{P}_{T,p}$. If instead $\mu=r/2$, then it follows
    \[
        \frac{r}{2} \le \frac{R\sqrt{\xi_{p,q}(d)}}{\sqrt{\lambda T} c_{\mu}(p,q,d)},
    \]
    which further implies
    \[
        \frac{r}{2}Tc_{\mu}(p,q,d)
        \le R\sqrt{\xi_{p,q}(d)T/\lambda}.
    \]
    Hence in both cases, it holds that
    \[
        GT\mu c_{\mu}(p,q,d) \le GR\sqrt{\xi_{p,q}(d) T/\lambda} = \mathcal{O}\left(G  \sqrt{\left(A_\psi^\alpha + G_\psi^\alpha \mathcal{P}_{T,p}\right)\xi_{p,q}(d) T  / \lambda}\right).
    \]
    Therefore, we obtain
    \begin{align*}
        \mE \left[\sum_{t=1}^T \frac{f_t(\bx^{t,+}) + f_t(\bx^{t,-})}{2} - f_t(\bu^t)\right]
        = \mathcal{O}\left(G  \sqrt{\left(A_\psi^\alpha + G_\psi^\alpha \mathcal{P}_{T,p}\right)\xi_{p,q}(d) T  / \lambda}\right).
    \end{align*}
    The proof is completed.
\end{proof}

\section{Parameter-free Bandit Mirror Descent}\label{section.pbmd}
\subsection{Algorithm design}

\begin{algorithm}[h]
  \caption{Parameter-free Bandit Mirror Descent (\texttt{PBMD})}
  \label{alg.PBMD}
  \begin{algorithmic}[1]
    \item \textbf{Input:} smoothing parameter $\mu>0$, shrinkage parameter $\alpha>0$, weight update parameter $\gamma > 0$, number of base learners $N$, candidate step sizes $\{\eta_{(k)}\}_{k=1}^{N}$;
    \item \textbf{Initialize:} pick $\by^1\in\mathcal{X}_{\alpha}$, set $\by_{(k)}^{1}=\by^1$ for every base learner $k$, and set $\omega_{(k)}^{1}= \frac{N+1}{N}\frac{1}{k(k+1)}$ for each base learner $k$;
    \item \textbf{For} $t = 1,2,\ldots,T$ \textbf{do}
    \item \quad Meta learner receives $\by_{(k)}^{t}$, $\forall k$, and combines the decision $\by^t = \sum_{k=1}^N \omega_{(k)}^{t}\by_{(k)}^{t}$; \label{line.meta combine}
    \item \quad Meta learner generates $\bs^t \sim \unif(\partial \mathbb{B}_1^d)$, and plays decisions: \label{line.meta play}
        \begin{align*}
            \bx^{t,+} = \by^t + \mu \bs^t, \quad \bx^{t,-} = \by^t - \mu \bs^t;
        \end{align*}
    \item \quad Meta learner receives $f_t(\bx^{t,+})$, $f_t(\bx^{t,-})$ as feedback, and constructs the  estimator:
                    \begin{align*}
                        \bg^t = \frac{d}{2\mu}\left(f_t(\bx^{t,+}) - f_t(\bx^{t,-})\right) \cdot \sign(\bs^t);
                    \end{align*}
    \item \quad Meta learner constructs the surrogate $\varphi^t(\by) = \inner{\bg^t}{\by - \by^t}$, and updates weights \label{line.update omega}
                    \begin{align*}
                        \omega_{(k)}^{t+1}=\frac{ \omega_{(k)}^{t} \exp \left(-\gamma \varphi^t(\by_{(k)}^{t})\right)}{\sum_{k = 1}^N  \omega_{(k)}^{t} \exp \left(-\gamma \varphi^t(\by_{(k)}^{t})\right)};
                    \end{align*}
    \item \quad Each base learner receives $\bg^t$, and updates the iterate in parallel:
    \begin{align*}
     \by_{(k)}^{t+1} = \argmin_{\by \in \mathcal{X}_{\alpha}}\inner{\nabla \varphi^t(\by_{(k)}^{t})}{\by} + \frac{1}{\eta_{(k)}} B_{\psi}(\by;\by_{(k)}^{t});
    \end{align*}\label{line:parallel update}
    \item \textbf{end For}
  \end{algorithmic}
\end{algorithm}

In this section, we introduce the parameter-free Bandit Mirror Descent (\texttt{PBMD}) in Algorithm \ref{alg.PBMD}, which eliminates the need for prior knowledge of path variation while attaining the same bound as stated in Theorem \ref{thm.dynamic regret of BMD}. The high-level idea is to maintain multiple candidates, referred to as \textit{base learners}, in parallel with a \textit{meta learner}. The meta learner employs expert-tracking techniques \citep{cesa2006prediction} to aggregate information, combine predictions, and adaptively track the best parameter (Lines \ref{line.meta combine} and \ref{line.meta play} in Algorithm \ref{alg.PBMD}). This allows the meta learner to make well-informed decisions in non-stationary environments.

Let us begin by removing the dependence of path variation on the step size. From Equation \eqref{eq.optimal step size}, although the exact value of $\mathcal{P}_{T,p}$ is unknown, the range of the optimal step size can still be determined. Given that $\|\bu^{t+1} - \bu^t\|_p \le \|\bu^{t+1} \|_p + \|\bu^t\|_p \le 2R$, we know that $\eta_\star \in [\eta_{\operatorname{lb}}, \eta_{\operatorname{ub}}]$ with
\begin{align*}
    \eta_{\operatorname{lb}} &\triangleq \sqrt{\frac{A_\psi^\alpha}{6(1+\sqrt{2})^2 G^2 \xi_{p,q}(d)T/\lambda}}, \\
    \eta_{\operatorname{ub}} &\triangleq \sqrt{\frac{A_\psi^\alpha + 2RG_\psi^\alpha T}{6(1+\sqrt{2})^2 G^2 \xi_{p,q}(d)T/\lambda}}.
\end{align*}
Hence, we can apply the grid search by constructing a pool of candidate step sizes, denoted by 
\begin{align*}
   \mathcal{E} = \{\eta_{(k)}: k = 1, \ldots, N \}. 
\end{align*}
The elements of this pool are given by
\begin{equation}\label{eq.candidate step size pool}
   \eta_{(k)} = 2^{k-1} \sqrt{\frac{A_\psi^\alpha}{6(1+\sqrt{2})^2 G^2 \xi_{p,q}(d)T/\lambda}}, 
\end{equation}
and $N$ is the number of candidate step sizes, set as
\begin{align*}
    N=\left\lceil\frac{1}{2} \log _2\left(1+\frac{2 R G_\psi^\alpha T}{A_\psi^\alpha}\right)\right\rceil+1.
\end{align*}
Note that all base learners are initialized at the same point $\by^1$. Maintain $N$ base learners in parallel and each runs \texttt{BMD} with a distinct step size in the pool $\mathcal{E}$ (Line \ref{line:parallel update} in Algorithm \ref{alg.PBMD}). There must exist a base learner $k^\star$, equipped with the step size $\eta_{(k^\star)}$ for some $k^\star \in \{1, \ldots, N\}$ such that $\eta_{(k^\star)} \le \eta_\star \le 2\eta_{(k^\star)}$ with
\begin{equation}\label{eq.upper bound k star}
    k^\star \le \left\lceil\frac{1}{2} \log _2\left(1+\frac{G_\psi^\alpha \mathcal{P}_{T,p}}{A_\psi^\alpha}\right)\right\rceil+1.
\end{equation}
It yields that the $k^\star$th base learner enjoys dynamic regret comparable to that of the optimal step size \eqref{eq.optimal step size}. Then the meta learner combines the predictions from all base learners and plays the decision accordingly.

While the above argument is promising, it has a minor limitation: maintaining $N$ base learners requires $\mathcal{O}(N)$ function evaluations per iteration, which is not allowed. In this scenario, only two function values can be observed at each iteration and are used to construct the gradient estimator \eqref{eq.gradient estimator}. However, we may replace the original functions with some suitable surrogates $\{\varphi^t\}_{t=1}^T$. Indeed, based on the three-term decomposition \eqref{eq.three terms decomposition} and Equation \eqref{eq.bound term a}, the main effort of proving the expected dynamic regret analysis is to derive an upper bound for \texttt{term (d)}. By using the convexity of smoothing functions, we derive the relation
\begin{align*}
   f^\mu_t(\by^t) - f^\mu_t(\bv^t) \le \inner{\nabla f^\mu_t(\by^t)}{\by^t - \bv^t} = \mE_{\bs^t}\left[ \inner{\bg^t}{\by^t - \bv^t} \mid \cF_{t-1}\right], \ \forall t = 1, \ldots, T.
\end{align*}
Then for any iteration $t$, it suggests that the linear functions $\varphi^t(\by) = \inner{\bg^t}{\by - \by^t}$ can serve as a surrogate since
\begin{align*}
   f^\mu_t(\by^t) - f^\mu_t(\bv^t) \le \mE_{\bs^t}\left[ \varphi^t(\by^t) - \varphi^t(\bv^t) \mid \cF_{t-1}\right].
\end{align*}
By summing over $t$ from $t = 1$ to $T$ and using tower  rule, we obtain
\begin{align*}
    \mE\left[\sum_{t=1}^T f^\mu_t(\by^t) - f^\mu_t(\bv^t)\right]  \le \mE\left[\sum_{t=1}^T \varphi^t(\by^t) - \varphi^t(\bv^t)\right].
\end{align*}
Therefore, we can now initialize $N$ base learners to perform \texttt{BMD} over surrogates $\{\varphi^t\}_{t=1}^T$ where each base learner is associated with a specific step size. Since the gradient of $\varphi^t$ is essentially $\bg^t$, it can be computed by querying the function value of the original loss function $f_t$ only twice. Finally, putting all components together, the \texttt{PBMD} algorithm is completed.

\subsection{Regret analysis of \texttt{PBMD}}
Recall the preliminary investigation of the expected dynamic regret in Section \ref{section.bmd}. By combining three-term decomposition \eqref{eq.three terms decomposition} and Lemma \ref{lemma.bound term b and c}, we obtain
\begin{align*}
     &\mE \left[\sum_{t=1}^T \frac{f_t(\bx^{t,+}) + f_t(\bx^{t,-})}{2} - f_t(\bu^t)\right] \\
   \le \  &\mE\left[\sum_{t=1}^T f^\mu_t(\by^t) - f^\mu_t(\bv^t)\right] + G T\mu c_{\mu}(p,q,d)
 \\
   \le \  &\mE\left[\sum_{t=1}^T \varphi^t(\by^t) - \varphi^t(\bv^t)\right] + G T\mu c_{\mu}(p,q,d).
\end{align*}
The last inequality holds due to the construction of surrogates. Hence, it is sufficient to analyze the expected dynamic regret over the surrogate. For any base learner $k \in \{1, 2, \ldots, N\}$, the regret can be decomposed as
\begin{equation}\label{eq.decomposition of surrogate dynamic regret}
   \mE\left[\sum_{t=1}^T \varphi^t(\by^t) - \varphi^t(\bv^t)\right] = \underbrace{\mE\left[\sum_{t=1}^T \varphi^t(\by_{(k)}^{t}) - \varphi^t(\bv^t)\right]}_{\texttt{base-regret}} + \underbrace{\mE\left[\sum_{t=1}^T \varphi^t(\by^t) - \varphi^t(\by_{(k)}^{t})\right]}_{\texttt{meta-regret}}.
\end{equation}
Notice that the \texttt{base-regret} corresponds to \texttt{term (d)} in \eqref{eq.bound term a}, but with the smoothing function $f^\mu_t$ replaced by the surrogate $\varphi^t$ for the $k$-th base learner. Since each base learner performs \texttt{BMD} with respect to $\varphi^t$, the result in Lemma \ref{lemma.pathwise mirror descent} can be applied. Furthermore, as Algorithm \ref{alg.PBMD} employs a grid search to determine the optimal step size $\eta_\star$, then the \texttt{base-regret} can be well bounded for some $k = k^\star$. 

\begin{lemma}\label{lemma.bound base}
    Suppose Assumptions \ref{assumption.bounded region} and \ref{assumption.G Lipschitz} hold. There exists a base learner $k^\star$ such that
    \begin{align*}
       \mE\left[\sum_{t=1}^T \varphi^t(\by_{(k^\star)}^{t}) - \varphi^t(\bv^t)\right]
       \le 3(1+\sqrt{2})G \sqrt{6\left(A_\psi^\alpha + G_\psi^\alpha \mathcal{P}_{T,p}\right)\xi_{p,q}(d) T/\lambda},
    \end{align*}
     where the expectation is taken with respect to all randomness $\bs^1, \ldots, \bs^T$.
\end{lemma}
\begin{proof}
    Fix a base learner $k$. Since $\nabla \varphi^t(\by)=\bg^t$ is constant in $\by$, the update in Line \ref{line:parallel update} of Algorithm \ref{alg.PBMD} is exactly mirror descent with step size $\eta_{(k)}$ and the vector sequence $\bg^1,\ldots,\bg^T$. Therefore, Lemma \ref{lemma.pathwise mirror descent} implies that
    \begin{align*}
        \sum_{t=1}^T \big(\varphi^t(\by_{(k)}^{t}) - \varphi^t(\bv^t)\big)
        = \sum_{t=1}^T \inner{\bg^t}{\by_{(k)}^{t} - \bv^t}
        \le \frac{F_\psi^\alpha + B_{\psi}(\bv^1;\by^1)}{\eta_{(k)}}
        + \frac{G_\psi^\alpha\mathcal{P}_{T,p}}{\eta_{(k)}}
        + \frac{\eta_{(k)}}{2\lambda}\sum_{t=1}^T \|\bg^t\|_{p^*}^2.
    \end{align*}
    Taking expectation and using \eqref{eq.definition A psi} together with \eqref{eq.varaince}, we obtain
    \begin{align*}
        \mE\left[\sum_{t=1}^T \varphi^t(\by_{(k)}^{t}) - \varphi^t(\bv^t)\right]
        \le \frac{A_\psi^\alpha + G_\psi^\alpha \mathcal{P}_{T,p}}{\eta_{(k)}}
        + \frac{6(1+\sqrt{2})^2 G^2 \xi_{p,q}(d) T}{\lambda}\eta_{(k)}.
    \end{align*}
    By construction of the pool \eqref{eq.candidate step size pool}, there exists $k^\star$ such that $\eta_{(k^\star)} \le \eta_\star \le 2\eta_{(k^\star)}$. Hence
    \begin{align*}
        \mE\left[\sum_{t=1}^T \varphi^t(\by_{(k^\star)}^{t}) - \varphi^t(\bv^t)\right]
        \le \frac{2(A_\psi^\alpha + G_\psi^\alpha \mathcal{P}_{T,p})}{\eta_\star}
        + \frac{6(1+\sqrt{2})^2 G^2 \xi_{p,q}(d) T}{\lambda}\eta_\star.
    \end{align*}
    Substituting the definition \eqref{eq.optimal step size} of $\eta_\star$ yields the claim.
\end{proof}

We analyze the meta-regret from an entropy mirror descent perspective, viewing the multiplicative weights update as online mirror descent on the simplex. This approach avoids the use of Hoeffding-type inequalities \citep{zhang2018adaptive,zhao2021bandit}, and removes the need for nonnegativity and exponential boundedness assumptions on the losses. Let
\[
   \Delta_N \triangleq \left\{\boldsymbol{\omega}\in \reals^N : \sum_{j=1}^N \omega_j = 1, \boldsymbol{\omega} \ge 0\right\},
\]
be the simplex, and define the Kullback-Leibler divergence
\[
   D_{\mathrm{KL}}(\bu\|\boldsymbol{\omega}) \triangleq \sum_{j=1}^N u_j \log\!\left(\frac{u_j}{\omega_j}\right),
\]
with the convention $0\log 0 = 0$. Note that the multiplicative weights update (Line \ref{line.update omega}) in Algorithm \ref{alg.PBMD} is equivalent to
\begin{equation}\label{eq:KL update}
    \boldsymbol{\omega}^{t+1}
   = \argmin_{\boldsymbol{\omega}\in \Delta_N}
     \gamma \inner{\boldsymbol{\ell}^t}{\boldsymbol{\omega}} + D_{\mathrm{KL}}(\boldsymbol{\omega}\|\boldsymbol{\omega}^t),
\end{equation}
where $\boldsymbol{\ell}^t$ is the realized loss vector at round $t$.

\begin{lemma}\label{lemma.pathwise entropic md}
    Let $\boldsymbol{\ell}^1,\ldots,\boldsymbol{\ell}^T \in \reals^N$ be arbitrary vectors, and let $\boldsymbol{\omega}^{t+1}$ be generated by update \eqref{eq:KL update}. Then for every $\bu \in \Delta_N$, we have
    \begin{equation*}
        \sum_{t=1}^T \inner{\boldsymbol{\ell}^t}{\boldsymbol{\omega}^t - \bu}
        \le \frac{D_{\mathrm{KL}}(\bu\|\boldsymbol{\omega}^1)}{\gamma}
        + \frac{\gamma}{2}\sum_{t=1}^T \|\boldsymbol{\ell}^t\|_{\infty}^2.
    \end{equation*}
\end{lemma}
\begin{proof}
    Let $\Psi(\boldsymbol{\omega}) \triangleq \sum_{j=1}^N \omega_j \log (\omega_j)$ be the negative entropy. Since $\boldsymbol{\omega}^1$ has strictly positive coordinates and the update is multiplicative, all iterates $\boldsymbol{\omega}^t$ remain in the interior of $\Delta_N$, so the gradients of the negative entropy $\Psi$ at the iterates are well defined. Fix an arbitrary sequence $\boldsymbol{\ell}^1,\ldots,\boldsymbol{\ell}^T$. By optimality of the update, we have
    \begin{align*}
        \inner{\gamma \boldsymbol{\ell}^t + \nabla \Psi(\boldsymbol{\omega}^{t+1}) - \nabla \Psi(\boldsymbol{\omega}^t)}{\bu - \boldsymbol{\omega}^{t+1}} \ge 0,
        \quad \forall \bu \in \Delta_N.
    \end{align*}
    Applying the three-point identity for the Bregman divergence of $\Psi$ gives
    \begin{align*}
        \gamma \inner{\boldsymbol{\ell}^t}{\boldsymbol{\omega}^t - \bu}
        \le \ B_{\Psi}(\bu;\boldsymbol{\omega}^t) - B_{\Psi}(\bu;\boldsymbol{\omega}^{t+1}) - B_{\Psi}(\boldsymbol{\omega}^{t+1};\boldsymbol{\omega}^t) + \gamma \inner{\boldsymbol{\ell}^t}{\boldsymbol{\omega}^t - \boldsymbol{\omega}^{t+1}}.
    \end{align*}
    Since the negative entropy is $1$-strongly convex with respect to $\ell_1$, we obtain
    \[
       B_{\Psi}(\boldsymbol{\omega}^{t+1};\boldsymbol{\omega}^t)
       = D_{\mathrm{KL}}(\boldsymbol{\omega}^{t+1}\|\boldsymbol{\omega}^t)
       \ge \frac{1}{2}\|\boldsymbol{\omega}^{t+1} - \boldsymbol{\omega}^t\|_1^2.
    \]
    By Hölder's inequality and Young's inequality, we have
    \begin{align*}
        \gamma \inner{\boldsymbol{\ell}^t}{\boldsymbol{\omega}^t - \boldsymbol{\omega}^{t+1}}
        \le \gamma \|\boldsymbol{\ell}^t\|_{\infty}\|\boldsymbol{\omega}^t - \boldsymbol{\omega}^{t+1}\|_1
        \le \frac{\gamma^2}{2}\|\boldsymbol{\ell}^t\|_{\infty}^2 + \frac{1}{2}\|\boldsymbol{\omega}^t - \boldsymbol{\omega}^{t+1}\|_1^2.
    \end{align*}
    Combining the above two inequalities yields
    \begin{align*}
        \gamma \inner{\boldsymbol{\ell}^t}{\boldsymbol{\omega}^t - \bu}
        \le B_{\Psi}(\bu;\boldsymbol{\omega}^t) - B_{\Psi}(\bu;\boldsymbol{\omega}^{t+1}) + \frac{\gamma^2}{2}\|\boldsymbol{\ell}^t\|_{\infty}^2.
    \end{align*}
    Summing over $t=1,\ldots,T$ and telescoping the Bregman terms gives
    \begin{align*}
        \sum_{t=1}^T \inner{\boldsymbol{\ell}^t}{\boldsymbol{\omega}^t - \bu}
        \le \frac{B_{\Psi}(\bu;\boldsymbol{\omega}^1)}{\gamma} + \frac{\gamma}{2}\sum_{t=1}^T \|\boldsymbol{\ell}^t\|_{\infty}^2.
    \end{align*}
    Since $B_{\Psi}(\bu;\boldsymbol{\omega}^1) = D_{\mathrm{KL}}(\bu\|\boldsymbol{\omega}^1)$, we complete the proof.
\end{proof}

\begin{lemma}\label{lemma.bound meta}
    Suppose Assumptions \ref{assumption.bounded region} and \ref{assumption.G Lipschitz} hold. For any base learner $k \in \{1,\ldots,N\}$ and any $\gamma > 0$, it holds that
    \begin{align*}
        \mE\left[\sum_{t=1}^T \varphi^t(\by^t) - \varphi^t(\by_{(k)}^{t})\right]
        \le \frac{-\log(\omega_{(k)}^{1})}{\gamma} + 24(1+\sqrt{2})^2 \gamma R^2 G^2 \xi_{p,q}(d)T.
    \end{align*}
    Consequently, if
    \[
        \gamma = \frac{1}{\sqrt{48}(1+\sqrt{2})RG\sqrt{\xi_{p,q}(d)T}},
    \]
    then we have
    \begin{align*}
        \mE\left[\sum_{t=1}^T \varphi^t(\by^t) - \varphi^t(\by_{(k)}^{t})\right]
        \le 7(1+\sqrt{2})RG\sqrt{\xi_{p,q}(d)T}\left(1-\log(\omega_{(k)}^{1})\right).
    \end{align*}
\end{lemma}
\begin{proof}
    For each round $t$, define the loss vector $\boldsymbol{\ell}^t \in \reals^N$ by
    \[
       \ell_j^t \triangleq \varphi^t(\by_{(j)}^{t}),
       \quad j=1,\ldots,N.
    \]
    Because $\varphi^t$ is linear and $\by^t = \sum_{j=1}^N \omega_{(j)}^{t}\by_{(j)}^{t}$, we have the following identity
    \begin{align*}
        \inner{\boldsymbol{\ell}^t}{\boldsymbol{\omega}^t}
        = \sum_{j=1}^N \omega_{(j)}^{t}\varphi^t(\by_{(j)}^{t})
        = \varphi^t\!\left(\sum_{j=1}^N \omega_{(j)}^{t}\by_{(j)}^{t}\right)
        = \varphi^t(\by^t)
        = 0.
    \end{align*}
    Applying Lemma \ref{lemma.pathwise entropic md} with $\bu=\mathbf e_k$ gives
    \begin{equation}\label{eq:pathwise inequality KL}
        \sum_{t=1}^T \big(\varphi^t(\by^t) - \varphi^t(\by_{(k)}^{t})\big)
        = \sum_{t=1}^T \inner{\boldsymbol{\ell}^t}{\boldsymbol{\omega}^t - \mathbf e_k}
        \le \frac{D_{\mathrm{KL}}(\mathbf e_k\|\boldsymbol{\omega}^1)}{\gamma}
        + \frac{\gamma}{2}\sum_{t=1}^T \|\boldsymbol{\ell}^t\|_{\infty}^2. 
    \end{equation}
    Since $D_{\mathrm{KL}}(\mathbf e_k\|\boldsymbol{\omega}^1) = -\log (\omega_{(k)}^{1})$, it remains to bound $\|\boldsymbol{\ell}^t\|_{\infty}^2$. For each $j$, since $\by_{(j)}^{t}, \by^t \in \mathcal{X}_{\alpha} \subseteq \mathcal{X} \subseteq R\mathbb{B}_p^d$, we have
    \begin{align*}
        |\ell_j^t|
        = |\varphi^t(\by_{(j)}^{t})|
        = |\inner{\bg^t}{\by_{(j)}^{t} - \by^t}|
        \le \|\bg^t\|_{p^*}\|\by_{(j)}^{t} - \by^t\|_p
        \le 2R\|\bg^t\|_{p^*}.
    \end{align*}
    It implies that $\|\boldsymbol{\ell}^t\|_{\infty}^2 \le 4R^2\|\bg^t\|_{p^*}^2$. Taking expectations and using \eqref{eq.varaince}, we obtain
    \begin{align*}
        \mE [\|\boldsymbol{\ell}^t\|_{\infty}^2]
        \le 4R^2 \mE\left[\mE\left[\|\bg^t\|_{p^*}^2 \mid \cF_{t-1}\right]\right]
        \le 48(1+\sqrt{2})^2 R^2 G^2 \xi_{p,q}(d).
    \end{align*}
    Taking expectation in the regret bound \eqref{eq:pathwise inequality KL} yields
    \begin{align*}
        \mE\left[\sum_{t=1}^T \varphi^t(\by^t) - \varphi^t(\by_{(k)}^{t})\right]
        \le \frac{-\log (\omega_{(k)}^{1})}{\gamma}
        + 24(1+\sqrt{2})^2 \gamma R^2 G^2 \xi_{p,q}(d)T,
    \end{align*}
    which proves the first result. The second claim follows by substituting the value of $\gamma$ and using $-\log (\omega_{(k)}^{1}) \ge 0$.
\end{proof}

\begin{theorem}\label{thm.dynamic regret of PBMD}
    Suppose Assumptions \ref{assumption.bounded region} and \ref{assumption.G Lipschitz} hold. Let $\psi: \operatorname{int}(\mathcal{X}) \rightarrow \reals$ be a divergence-generating function that is $\lambda$-strongly convex with respect to the $\ell_p$-norm ($\lambda > 0$). Choose the step-size pool $\mathcal{E}$ as in \eqref{eq.candidate step size pool}, and set
    \begin{align*}
        \gamma &= \frac{1}{\sqrt{48}(1+\sqrt{2})RG\sqrt{\xi_{p,q}(d)T}}, \\
        \mu &= \min\left\{\frac{R\sqrt{\xi_{p,q}(d)}}{\sqrt{\lambda T} c_{\mu}(p,q,d)},\ \frac{r}{2}\right\},
    \end{align*}
    together with $\alpha = \mu/r$. Then the expected dynamic regret of \texttt{PBMD} satisfies
    \begin{align*}
        \mE \left[\sum_{t=1}^T \frac{f_t(\bx^{t,+}) + f_t(\bx^{t,-})}{2} - f_t(\bu^t)\right]
        = \mathcal{O}\left(G\sqrt{\left(A_\psi^\alpha + G_\psi^\alpha\mathcal{P}_{T,p}\right)\xi_{p,q}(d)T/\lambda}\right),
    \end{align*}
    where the expectation is taken with respect to all randomness $\bs^1, \ldots, \bs^T$.
\end{theorem}
\begin{proof}
    Apply \eqref{eq.decomposition of surrogate dynamic regret} with $k = k^\star$. Lemmas \ref{lemma.bound base} and \ref{lemma.bound meta} give
    \begin{equation}\label{eq.decomposition-bound}
        \begin{aligned}
            &\mE \left[\sum_{t=1}^T \frac{f_t(\bx^{t,+}) + f_t(\bx^{t,-})}{2} - f_t(\bu^t)\right] \\
        \le \ &\mathcal{O}\left(G\sqrt{\left(A_\psi^\alpha + G_\psi^\alpha\mathcal{P}_{T,p}\right)\xi_{p,q}(d)T/\lambda}\right)
        + 7(1+\sqrt{2})RG\sqrt{\xi_{p,q}(d)T}\left(1-\log\omega_{(k^\star)}^{1}\right)
        + G T\mu c_{\mu}(p,q,d).
        \end{aligned}
    \end{equation}
    For the \texttt{meta-regret}, we have
    \begin{align*}
        - \log \omega_{(k^\star)}^{1}
        = \log\left(\frac{N}{N+1}k^\star(k^\star+1)\right)
        \le 2\log(k^\star+1)
    \end{align*}
    due to that $\omega_{(k)}^{1}= \frac{N+1}{N}\frac{1}{k(k+1)}$. By \eqref{eq.upper bound k star}, we obtain
    \begin{align*}
        - \log \omega_{(k^\star)}^{1}
        \le 2\log\left(\left\lceil\frac{1}{2} \log _2\left(1+\frac{G_\psi^\alpha \mathcal{P}_{T,p}}{A_\psi^\alpha}\right)\right\rceil+2\right)
        = \mathcal{O}\left(1+\log\log\left(2+\frac{G_\psi^\alpha\mathcal{P}_{T,p}}{A_\psi^\alpha}\right)\right).
    \end{align*}
    Since $A_\psi^\alpha \ge \lambda R^2$ by \eqref{eq.definition A psi}, it follows that $RG\sqrt{\xi_{p,q}(d)T} \le G\sqrt{A_\psi^\alpha\xi_{p,q}(d)T/\lambda}$. Moreover, note that $1+\log\log(2+z)=\mathcal{O}(\sqrt{1+z})$ for any $z\ge 0$. Consequently, the \texttt{meta-regret} is bounded by
    \begin{equation}\label{eq.bound-meta}
         \mathcal{O}\left(G\sqrt{A_\psi^\alpha\xi_{p,q}(d)T/\lambda}\cdot \sqrt{1+\frac{G_\psi^\alpha\mathcal{P}_{T,p}}{A_\psi^\alpha}}\right)
        = \mathcal{O}\left(G\sqrt{\left(A_\psi^\alpha + G_\psi^\alpha\mathcal{P}_{T,p}\right)\xi_{p,q}(d)T/\lambda}\right).
    \end{equation}
    For the term $GT\mu c_\mu(p,q,d)$, if the minimum defining $\mu$ is attained at the first value, then we have $ GT\mu c_{\mu}(p,q,d) = GR\sqrt{\xi_{p,q}(d)T/\lambda}$, which is dominated by the first term because $R^2 \le A_\psi^\alpha \le A_\psi^\alpha + G_\psi^\alpha\mathcal{P}_{T,p}$. If instead $\mu=r/2$, then it holds that
    \[
        \frac{r}{2} \le \frac{R\sqrt{\xi_{p,q}(d)}}{\sqrt{\lambda T} c_{\mu}(p,q,d)},
    \]
    which further implies
    \begin{equation}\label{eq.biased-error}
        GT\mu c_{\mu}(p,q,d) \le \frac{Gr}{2}Tc_{\mu}(p,q,d) \le GR\sqrt{\xi_{p,q}(d)T/\lambda}.
    \end{equation}
    Substituting both bounds \eqref{eq.bound-meta} and \eqref{eq.biased-error} into \eqref{eq.decomposition-bound} completes the proof.
\end{proof}

\subsection{Examples: Euclidean space and cross-polytope}
We now present examples with concrete bounds for \texttt{PBMD}. For each corollary, we specify the choices of $p,q$, the domain $\mathcal{X}$, and the regularizer $\psi$, and then estimate the deterministic quantities $A_\psi^\alpha$, $G_\psi^\alpha$, $\lambda$, and $\xi_{p,q}(d)$ appearing in Theorem \ref{thm.dynamic regret of PBMD}. We begin with the standard Euclidean geometry, where $p = q = 2$ and $\psi(\bx)=\|\bx\|_2^2 / 2$.

\begin{corollary}[Euclidean space]\label{proposition.BMD 2 norm}
   Suppose Assumption \ref{assumption.G Lipschitz} holds. Let $p = q = 2$, $\mathcal{X} = \{\bx \in \reals^d: \|\bx\|_2 \le 1\}$, and $\psi(\bx) = \|\bx\|_2^2 / 2$. Then the expected dynamic regret of \texttt{PBMD} satisfies
   \begin{equation*}
            \mE \left[\sum_{t=1}^T \frac{f_t(\bx^{t,+}) + f_t(\bx^{t,-})}{2} - f_t(\bu^t)\right] = \mathcal{O}\left(G \sqrt{dT\left(1 + \mathcal{P}_{T, 2}\right)}\right).
   \end{equation*}
   under the parameter setting in Theorem \ref{thm.dynamic regret of PBMD}, where the expectation is taken with respect to all randomness $\bs^1, \ldots, \bs^T$.
\end{corollary}
\begin{proof}
    Choose $\by^1 = 0$. The set $\mathcal{X} = \{\bx \in \reals^d: \|\bx\|_2 \le 1\}$ satisfies Assumption \ref{assumption.bounded region} with $r=R=1$. Moreover, we have
    \begin{align*}
        F_\psi^\alpha &\le \sup_{\bx \in \mathcal{X}} \frac{\|\bx\|_2^2}{2} - \inf_{\bx \in \mathcal{X}} \frac{\|\bx\|_2^2}{2} = \frac{1}{2}, \\
        \sup_{\bv \in \mathcal{X}} B_{\psi}(\bv;\by^1) &= \sup_{\bv \in \mathcal{X}} \frac{\|\bv\|_2^2}{2} \le \frac{1}{2}, \\
        G_\psi^\alpha &\le \sup_{\bx \in \mathcal{X}} \|\nabla \psi(\bx)\|_2 = \sup_{\bx \in \mathcal{X}} \|\bx\|_2 \le 1, \\
        \xi_{2,2}(d) &= d, \quad \lambda = 1.
    \end{align*}
    Hence we may take $A_\psi^\alpha=2$. Substituting these estimates into Theorem \ref{thm.dynamic regret of PBMD} yields
    \[
        \mE \left[\sum_{t=1}^T \frac{f_t(\bx^{t,+}) + f_t(\bx^{t,-})}{2} - f_t(\bu^t)\right]
        = \mathcal{O}\left(G \sqrt{dT\left(1 + \mathcal{P}_{T, 2}\right)}\right).
    \]
\end{proof}

The above dynamic regret of \texttt{PBMD} in Euclidean geometry is optimal, matching the minimax lower bound established in \cite{zhao2021bandit}. Our result improves the dynamic regret bound of the parameter-free bandit gradient descent by a factor of $\mathcal{O}(\sqrt{d})$. Specifically, notice that $\mathcal{P}_{T, 2} = 0$ when the comparator sequence is chosen as $\bu^1 = \ldots = \bu^T = \argmin_{\bx \in \mathcal{X}} \ \sum_{t=1}^T f_t(\bx)$. In this case, the above bound simplifies to $\mathcal{O}(G\sqrt{dT})$, which matches the optimal result established by \citet{shamir2017optimal} for static regret. Next, we consider online learning over a $d$-dimensional cross-polytope. In this setting we take $p = q = 1 + 1/\log(d)$ and use the regularizer $\psi(\bx) = \|\bx\|_p^2 / 2$.

\begin{corollary}[Cross-polytope]\label{proposition.BMD cross polytope}
   Suppose Assumption \ref{assumption.G Lipschitz} holds. Let $p = q = 1 + 1/\log(d)$, $\mathcal{X} = \{\bx \in \reals^d: \|\bx\|_1 \le 1\}$, and $\psi(\bx) = \|\bx\|_p^2 / 2$. Then the expected dynamic regret of \texttt{PBMD} satisfies
   \begin{equation*}
      \mE \left[\sum_{t=1}^T \frac{f_t(\bx^{t,+}) + f_t(\bx^{t,-})}{2} - f_t(\bu^t)\right] = \mathcal{O}\left(G \sqrt{d\log(d)T\left(1 + \mathcal{P}_{T,p}\right)}\right).
   \end{equation*}
    under the parameter setting in Theorem \ref{thm.dynamic regret of PBMD}, where the expectation is taken with respect to all randomness $\bs^1, \ldots, \bs^T$.
\end{corollary}
\begin{proof}
    Choose $\by^1=0$. Since $\|\bx\|_p \le \|\bx\|_1$ for every $\bx \in \reals^d$, we have $\mathcal{X} \subseteq \mathbb{B}_p^d$. Conversely, Lemma \ref{lemma.vector norm} implies $\|\bx\|_1 \le d^{1-1/p}\|\bx\|_p$, so $d^{1/p-1}\mathbb{B}_p^d \subseteq \mathcal{X}$. Hence Assumption \ref{assumption.bounded region} holds with $ r = d^{1/p-1}$, $R = 1$. Next, note that
    \begin{align*}
        F_\psi^\alpha &\le \frac{1}{2}, \quad \sup_{\bv \in \mathcal{X}} B_{\psi}(\bv;\by^1) = \sup_{\bv \in \mathcal{X}} \frac{\|\bv\|_p^2}{2} \le \frac{1}{2}.
    \end{align*}
    For $\psi(\bx)=\|\bx\|_p^2/2$ with $p \in (1,2)$, a direct calculation gives
    \[
       \nabla \psi(\bx) = \|\bx\|_p^{2-p}\big(|x_1|^{p-2}x_1,\ldots,|x_d|^{p-2}x_d\big),
    \]
    which implies
    \begin{align*}
        \|\nabla \psi(\bx)\|_{p^*}
        = \|\bx\|_p^{2-p}\left(\sum_{j=1}^d |x_j|^{(p-1)p^*}\right)^{1/p^*}
        = \|\bx\|_p^{2-p}\left(\sum_{j=1}^d |x_j|^{p}\right)^{1/p^*}
        = \|\bx\|_p.
    \end{align*}
    Since $\|\bx\|_p \le \|\bx\|_1 \le 1$ on $\mathcal{X}$, we get $G_\psi^\alpha \le 1$. Moreover, we have
    \[
        \xi_{p,q}(d) = d^{1 + \frac{2}{p} - \frac{2}{p}} = d,
        \quad
        \lambda = p-1 = \frac{1}{\log(d)},
    \]
    since $\psi(\bx)=\|\bx\|_p^2/2$ is $(p-1)$-strongly convex with respect to $\ell_p$-norm. Hence we may again take $A_\psi^\alpha=2$. Substituting these estimates into Theorem \ref{thm.dynamic regret of PBMD} yields
    \[
        \mE \left[\sum_{t=1}^T \frac{f_t(\bx^{t,+}) + f_t(\bx^{t,-})}{2} - f_t(\bu^t)\right]
        = \mathcal{O}\left(G \sqrt{d\log(d)T\left(1 + \mathcal{P}_{T,p}\right)}\right).
    \]
\end{proof}

\section{The Simplex Geometry}\label{section.simplex}
In this section, we study parameter-free bandit mirror descent over the simplex $ \Delta_d \triangleq \{\bx \in \reals^d : \bx \ge 0,\ \sum_{j=1}^d x_j = 1\}$. Since the simplex does not satisfy Assumption \ref{assumption.bounded region}, we treat this geometry separately. Fix the smoothing parameter $\mu$ used below. We assume that the adversary chooses functions $f_t$ that are convex, and $G$-Lipschitz with respect to $\ell_1$ on an open convex set $\mathcal U\subseteq \reals^d$. The comparator sequence is constrained to $\Delta_d$, but the learner is allowed to query $f_t$ at the two perturbed points generated by the algorithm. More precisely, after setting $\alpha=\mu$, define the shrunken simplex
\[
   \mathcal X_\alpha \triangleq \{(1-\alpha)\bx+\alpha\bc:\bx\in\Delta_d\}, \quad \bc\triangleq (1/d,\ldots,1/d)^\top .
\]
We require the following condition
\begin{equation}\label{eq.simplex extension}
   \mathcal X_\alpha + \mu \mathbb B_1^d \subseteq \mathcal U,
\end{equation}
which guarantees that both the queried points $\by^t\pm\mu\bs^t$ and the smoothing averages $\by+\mu\mathbb B_1^d$ are well defined for all $\by\in\mathcal X_\alpha$. For any comparator sequence $\bu^1,\ldots,\bu^T \in \Delta_d$,  we define the scaled comparator $\bv^t = (1-\alpha)\bu^t + \alpha \bc \in \mathcal{X}_\alpha$. Then the scaled comparators satisfy
\begin{align*}
    \sum_{t=1}^{T-1} \|\bv^{t+1} - \bv^t\|_1 = (1-\alpha)\sum_{t=1}^{T-1} \|\bu^{t+1} - \bu^t\|_1 \le \sum_{t=1}^{T-1} \|\bu^{t+1} - \bu^t\|_1 \le \mathcal{P}_{T,1}.
\end{align*}
Moreover, using the $\ell_1$-diameter of $\Delta_d$, we have
\begin{equation}\label{eq.simplex-comparator-bias}
   |f_t(\bv^t)-f_t(\bu^t)|
   \le G\|\bv^t-\bu^t\|_1
   = G\alpha\|\bc-\bu^t\|_1
   \le 2G\alpha,
\end{equation}
where we used that $\|\bc - \bu^t\|_1 \le 2$. Finally, every $\by \in \mathcal{X}_\alpha$ satisfies $y_j \ge \alpha/d$ for all $j$, so the entropy regularizer $\psi(\bx)=\sum_{j=1}^d x_j\log x_j$ is well-defined and continuously differentiable on $\mathcal{X}_\alpha$, and in particular at all iterates and scaled comparators. Recall that the proofs of Lemmas \ref{lemma.pathwise mirror descent}, 
\ref{lemma.bound base}, and \ref{lemma.bound meta} rely only on the first-order optimality condition over a convex update set, the differentiability of the regularizer at the iterates, and its strong convexity on that set. These properties hold on the relative interior of the simplex, and in particular on $\mathcal{X}_\alpha$, so the arguments of these lemmas remain valid in the current setting. For the negative entropy regularizer, it is well known that $B_\psi(\bx;\by)\ge \frac{1}{2}\|\bx-\by\|_1^2$, $\forall \bx,\by\in\Delta_d$, which shows that $\psi$ is $1$-strongly convex with respect to the $\ell_1$ norm under the condition used in Lemma \ref{lemma.pathwise mirror descent}.

\begin{theorem}[Simplex]\label{proposition.BMD simplex with entropy}
   Consider the simplex model above and assume \eqref{eq.simplex extension}. Let $d,T\ge3$, $p=q=1$, and run \texttt{PBMD} on $\mathcal X_\alpha$ with common initialization $\by^1=\bc$, entropy regularizer $\psi(\bx)=\sum_{j=1}^d x_j\log x_j$, and $\alpha=\mu$. Define
   \[
       A_\Delta \triangleq 2\log(d)+1,
       \quad
       c_\Delta\triangleq 3+2\zeta_1(d)\le 5,
   \]
   and choose
   \[
       \mu=\min\left\{\frac{\sqrt{A_\Delta d}}{\sqrt T\,c_\Delta},\frac12\right\},
       \quad
       \gamma=\frac{1}{\sqrt{48}(1+\sqrt2)G\sqrt{dT}} .
   \]
   The step-size pool is
   \[
      \eta_{(k)}=2^{k-1}
      \sqrt{\frac{A_\Delta}{6(1+\sqrt2)^2G^2dT}},
      \quad k=1,\ldots,N,
   \]
   where
   \[
      N=\left\lceil\frac12\log_2\left(1+\frac{2G_\Delta^\alpha T}{A_\Delta}\right)\right\rceil+1,
      \quad
      G_\Delta^\alpha\triangleq 1+\log(d/\alpha).
   \]
   Then, for every comparator sequence $\bu^1,\ldots,\bu^T\in\Delta_d$, we have
   \[
      \mE\left[
      \sum_{t=1}^T
      \frac{f_t(\bx^{t,+})+f_t(\bx^{t,-})}{2}-f_t(\bu^t)
      \right]
      =\mathcal O\left(
      G\sqrt{d\log(d)T\log(T)\left(1+\mathcal P_{T,1}\right)}
      \right),
   \]
   where the expectation is taken with respect to all randomness $\bs^1,\ldots,\bs^T$.
\end{theorem}
\begin{proof}
    For the entropy regularizer and the initialization $\by^1=\bc$, we have
    \[
        F_\psi^\alpha\le \log(d),
        \quad
        \sup_{\bv\in\mathcal X_\alpha}B_\psi(\bv;\bc)
        \le \sup_{\bv\in\Delta_d}\sum_{j=1}^d v_j\log\frac{v_j}{1/d}
        \le \log(d).
    \]
    Since $\Delta_d\subseteq \mathbb B_1^d$, the radius term in \eqref{eq.definition A psi} is $R^2=1$, so the choice $A_\Delta \triangleq 2\log(d)+1$ satisfies the analogue of \eqref{eq.definition A psi}. Also, for every $\by\in\mathcal X_\alpha$, it holds that
    \[
       \|\nabla\psi(\by)\|_\infty
       =\max_{1\le j\le d}|1+\log y_j|
       \le 1+\log(d/\alpha) \triangleq G_\Delta^\alpha .
    \]
    The second-moment bound in Lemma \ref{lemma.properties of gradient estimator} gives $\xi_{1,1}(d)=d$, and the entropy is $1$-strongly convex with respect to $\ell_1$ on the simplex.

    The regret decomposition is the same as \eqref{eq.three terms decomposition}, except that the comparator shrinkage is centered at $\bc$. Applying the proofs of Lemmas \ref{lemma.bound base} and \ref{lemma.bound meta} with $A_\Delta$, $G_\Delta^\alpha$, $\xi_{1,1}(d)=d$, and $\lambda=1$ yields
    \begin{equation}\label{eq.simplex-intermediate-bound}
        \mE\left[
        \sum_{t=1}^T
        \frac{f_t(\bx^{t,+})+f_t(\bx^{t,-})}{2}-f_t(\bu^t)
        \right]
        \le
        C_\Delta G\sqrt{dT\left(A_\Delta+G_\Delta^\alpha\mathcal P_{T,1}\right)}
        +GT\mu c_\Delta
    \end{equation}
    for some constant $C_\Delta>0$. For the bias term $G T \mu c_\Delta$, the chosen value of $\mu$ ensures that it is bounded by a constant multiple of $G\sqrt{A_\Delta d T}$. Indeed, if the first term in the definition of $\mu$ is attained, then $G T \mu c_\Delta = G\sqrt{A_\Delta d T}$. On the other hand, if $\mu = 1/2$, then $\sqrt{A_\Delta d}/(\sqrt{T} c_\Delta) > 1/2$, which implies $T c_\Delta \le 2\sqrt{A_\Delta d T}$ and hence $G T \mu c_\Delta \le G\sqrt{A_\Delta d T}$. Therefore, in all cases, the bias term is of order $\mathcal{O}(G\sqrt{A_\Delta d T})$. Finally, because $\alpha=\mu$ and $d,T\ge3$, the above choice of $\mu$ implies $G_\Delta^\alpha=1+\log(d/\mu)=\mathcal O(\log(dT))$. We thus have
    \[
        A_\Delta+G_\Delta^\alpha\mathcal P_{T,1}
        =\mathcal O\left(\log(d)+\mathcal P_{T,1}\log(dT)\right).
    \]
    Since $\log(dT)\le 2\log(d)\log(T)$ and $\log(d)\le \log(d)\log(T)$ for $d,T\ge3$, we obtain
    \begin{align*}
        \mE\left[
      \sum_{t=1}^T
      \frac{f_t(\bx^{t,+})+f_t(\bx^{t,-})}{2}-f_t(\bu^t)
      \right]
      =\mathcal O\left(
      G\sqrt{d\log(d)T\log(T)\left(1+\mathcal P_{T,1}\right)}
      \right).
    \end{align*}
    The proof is completed.
\end{proof}

We close this section by proving a lower bound for the simplex geometry. The lower bound is stated for the two-point feedback model on the simplex, in which the learner queries two points
\(\bx^{t,+},\bx^{t,-}\in\Delta_d\) at each round. The proof proceeds in two steps. First, we establish a lower bound on static regret for the bandit algorithm with two-point feedback on the simplex. We then convert this static bound into a dynamic one via a blocking construction, following the approach in \cite{zhao2021bandit}. For any path variation $\cP_{T,1} \ge 0$, define the comparator class
\[
   \mathcal C_T^\Delta(\mathcal P_{T,1})
   \triangleq
   \left\{
      \bu^1,\ldots,\bu^T\in\Delta_d:
      \sum_{t=1}^{T-1}\|\bu^{t+1}-\bu^t\|_1\le \mathcal P_{T,1}
   \right\}.
\]
Let \(\mathcal F_\Delta(G)\) denote the class of convex functions on \(\Delta_d\) that are \(G\)-Lipschitz with respect to \(\|\cdot\|_1\). We define the minimax dynamic regret over the simplex as
\begin{equation}\label{eq.simplex-minimax-dynamic-regret}
   \mathfrak R_T^\Delta(\mathcal P_{T,1}) \triangleq \inf_{\mathcal A} \sup_{f_1,\ldots,f_T\in\mathcal F_\Delta(G)} \mE\left[ \sum_{t=1}^T \frac{f_t(\bx^{t,+})+f_t(\bx^{t,-})}{2} - \inf_{\bu^1,\ldots,\bu^T\in\mathcal C_T^\Delta(\mathcal P_{T,1})} \sum_{t=1}^T f_t(\bu^t) \right],
\end{equation}
where the infimum is over all possibly randomized two-point bandit algorithms whose queries belong to \(\Delta_d\), and the expectation is with respect to all randomness of the algorithm. We now establish the lower bound on static regret for the bandit algorithm with two-point feedback on the simplex. 

\begin{lemma}\label{lem:static-simplex-lower}
There exists a constant \(c_1>0\) such that, for any
\(n\ge d\ge7\), it holds that
\[
   \inf_{\mathcal A} \sup_{f_1,\ldots,f_n\in\mathcal F_\Delta(G)} \mE \left[\sum_{t=1}^n \frac{f_t(\bx^{t,+})+f_t(\bx^{t,-})}{2} - \min_{\bx\in\Delta_d} \sum_{t=1}^n f_t(\bx) \right] \ge c_1G\sqrt{\frac{dn}{\log d}}.
\]
\end{lemma}
\begin{proof}
Let $m\triangleq \left\lfloor (d-1)/2\right\rfloor$. Since \(d\ge7\), we have \(m\ge3\). For any \(\bx\in\Delta_d\), define the projection map $\pi_{d,m}:\reals^d\to\reals^m$ as
\[
   \bigl(\pi_{d,m}(\bx)\bigr)_i \triangleq x_i-x_{m+i}, \quad \forall i=1,\ldots,m.
\]
We first show that $\pi_{d,m}(\Delta_d)=\mathbb B_1^m$. Note that
\begin{align*}
	\|\pi_{d,m}(\bx)\|_1 = \sum_{i=1}^m |x_i-x_{m+i}| \le \sum_{i=1}^m (x_i+x_{m+i}) \le 1,
\end{align*}
which implies that \(\pi_{d,m}(\Delta_d)\subseteq\mathbb B_1^m\). Conversely, for any \(\bz\in\mathbb B_1^m\), define the embedding map $\iota_{d,m}:\reals^m \to\reals^d$ as 
\begin{align*}
	\iota_{d,m}(\bz)
   \triangleq
   \left(
      z_1^+,\ldots,z_m^+,
      z_1^-,\ldots,z_m^-,
      1-\|\bz\|_1,
      0,\ldots,0
   \right)^\top
   \in\reals^d,
\end{align*}
where $z_i^+\triangleq \max\{z_i,0\}$ and $ z_i^-\triangleq \max\{-z_i,0\}$. Since \(\sum_{i=1}^m(z_i^++z_i^-)=\|\bz\|_1\le1\), we have \(\iota_{d,m}(\bz)\in\Delta_d\). Moreover, note that $\pi_{d,m}\bigl(\iota_{d,m}(\bz)\bigr)=\bz$, which yields that
\begin{equation}\label{eq:pi-simplex-surjective}
   \pi_{d,m}(\Delta_d)=\mathbb B_1^m .
\end{equation}
Let \(\mathcal F_{\mathbb B_1^m}(G)\) be the class of convex functions on \(\mathbb B_1^m\) that are \(G\)-Lipschitz with respect to \(\|\cdot\|_1\).  Now fix any \(h_1,\ldots,h_n\in\mathcal F_{\mathbb B_1^m}(G)\).
For each \(t\), define the induced simplex loss
\[
   f_t(\bx) \triangleq h_t\bigl(\pi_{d,m}(\bx)\bigr), \quad \forall \bx \in\Delta_d.
\]
Since \(\pi_{d,m}\) is linear and \(h_t\) is convex, \(f_t\) is convex on
\(\Delta_d\). Consequently, for any \(\bx,\by\in\Delta_d\), we have
\[
\begin{aligned}
   |f_t(\bx)-f_t(\by)| \le \ &G\|\pi_{d,m}(\bx)-\pi_{d,m}(\by)\|_1 \\
   = \ &G\sum_{i=1}^m |(x_i-y_i)-(x_{m+i}-y_{m+i})| \\
   \le \ &G\sum_{i=1}^m \bigl(|x_i-y_i|+|x_{m+i}-y_{m+i}|\bigr) \\
   \le \ &G\|\bx-\by\|_1,
\end{aligned}
\]
which implies that\(f_t\in\mathcal F_\Delta(G)\).

Fix any bandit algorithm with two-point feedback for the simplex problem. We construct a bandit algorithm with two-point feedback  for the \(\ell_1\)-ball problem as follows. Whenever the bandit algorithm on the simplex queries \(\bx^{t,+},\bx^{t,-}\in\Delta_d\), the bandit algorithm on the \(\ell_1\)-ball algorithm queries $\bz^{t,+} = \pi_{d,m}(\bx^{t,+})$ and $\bz^{t,-} = \pi_{d,m}(\bx^{t,-})$. Besides, the observed values are exactly the same as $f_t(\bx^{t,+})=h_t(\bz^{t,+})$ and $ f_t(\bx^{t,-})=h_t(\bz^{t,-})$. By \eqref{eq:pi-simplex-surjective}, it follows that
\begin{align*}
	 \min_{\bx\in\Delta_d} \sum_{t=1}^n f_t(\bx) = \min_{\bx\in\Delta_d} \sum_{t=1}^n h_t\bigl(\pi_{d,m}(\bx)\bigr) = \min_{\bz\in\mathbb B_1^m} \sum_{t=1}^n h_t(\bz).
\end{align*}
Hence every bandit algorithm with two-point feedback on simplex algorithm induces an \(\ell_1\)-ball algorithm with the same static regret on the corresponding losses. For the two-point feedback model, based on the Proposition~2 of \cite{duchi2015optimal}, for any algorithm, we can always find a sequence of loss functions $h_1, \ldots, h_n$ such that
\begin{align*}
   \inf_{\mathcal A} \sup_{h_1,\ldots,h_n\in\mathcal F_{\mathbb B_1^m}(G)} \mE \left[\sum_{t=1}^n \frac{h_t(\bz^{t,+})+h_t(\bz^{t,-})}{2} - \min_{\bz\in\mathbb B_1^m} \sum_{t=1}^n h_t(\bz) \right] \ge c_0 G\sqrt{\frac{mn}{\log(3m)}},
\end{align*}
where $c_0 > 0$ is a constant. Therefore, we have
\[
   \inf_{\mathcal A}
   \sup_{f_1,\ldots,f_n\in\mathcal F_\Delta(G)}
   \mE
   \left[
      \sum_{t=1}^n
      \frac{f_t(\bx^{t,+})+f_t(\bx^{t,-})}{2}
      -
      \min_{\bx\in\Delta_d}
      \sum_{t=1}^n f_t(\bx)
   \right]
   \ge
   c_0G\sqrt{\frac{mn}{\log(3m)}} .
\]
Since \(m=\lfloor(d-1)/2\rfloor\), for \(d\ge7\), we have \(m=\Theta(d)\) and \(\log(3m)=\Theta(\log d)\). Thus there exists a universal constant \(c_1>0\) such that
\[
   c_0G\sqrt{\frac{mn}{\log(3m)}} \ge c_1G\sqrt{\frac{dn}{\log d}} .
\]
The proof is completed.
\end{proof}

We now prove the dynamic regret lower bound by a blocking argument. Each block carries
an independent copy of the hard static simplex instance from Lemma~\ref{lem:static-simplex-lower}. The comparator is constant inside each block and changes only at block boundaries.
\begin{theorem}\label{thm:simplex-dynamic-lower}
There exists a constant \(c>0\) such that for any \(d\ge7\) and any path variation satisfying $0 \le \mathcal P_{T,1}\le T/8d -1$, we have
\[
   \mathfrak R_T^\Delta(\mathcal P_{T,1}) \ge cG \sqrt{\frac{dT(1+\mathcal P_{T,1})}{\log d}} .
\]
\end{theorem}
\begin{proof}
Let $ K \triangleq \left\lfloor \mathcal P_{T,1}/2 \right\rfloor+1$. Since $1+\mathcal P_{T,1}\le T/(8d)$, we have $K\le 1+\mathcal P_{T,1}\le T/(8d)$, which implies that $n\triangleq \left\lfloor T/K \right\rfloor \ge d$. We partition the first \(Kn\) rounds into \(K\) consecutive blocks
\[
   \mathcal I_i \triangleq \{(i-1)n+1,\ldots,in\}, \quad i=1,\ldots,K .
\]
The remaining \(T-Kn\) rounds are assigned the zero loss function. For each block \(\mathcal I_i\), we draw an independent hard static simplex instance from Lemma~\ref{lem:static-simplex-lower}. For the losses in block \(i\), we choose
\[
   \bu^{(i)}
   \in
   \arg\min_{\bu\in\Delta_d}
   \sum_{t\in\mathcal I_i} f_t(\bu).
\]
We define the comparator sequence by $\bu^t=\bu^{(i)}$ for all $t\in\mathcal I_i$, and set \(\bu^t=\bu^{(K)}\) on the remaining \(T-Kn\) rounds. Since the \(\ell_1\)-diameter of \(\Delta_d\) is at most \(2\), we have $\|\bu^{(i+1)}-\bu^{(i)}\|_1\le2$, $\forall i=1,\ldots,K-1$. Consequently, it follows that
\begin{align*}
       \sum_{t=1}^{T-1}\|\bu^{t+1}-\bu^t\|_1 = \sum_{i=1}^{K-1} \|\bu^{(i+1)}-\bu^{(i)}\|_1 \le 2(K-1) \le \mathcal P_{T,1},
\end{align*}
which validates that \(\bu^1,\ldots,\bu^T\in
\mathcal C_T^\Delta(\mathcal P_{T,1})\).

Recalling that $n \ge d$, Lemma~\ref{lem:static-simplex-lower} can be applied on each block. Thus, we have
\[
   \mE\left[ \sum_{t\in\mathcal I_i} \frac{f_t(\bx^{t,+})+f_t(\bx^{t,-})}{2} - \min_{\bu\in\Delta_d} \sum_{t\in\mathcal I_i}f_t(\bu) \right] \ge c_1G\sqrt{\frac{dn}{\log d}}.
\]
Summing over the blocks and using the constructed comparator sequence yields
\[
   \mathfrak R_T^\Delta(\mathcal P_{T,1})
   \ge
   Kc_1G\sqrt{\frac{dn}{\log d}} .
\]
Given \(n=\lfloor T/K\rfloor\) and \(K\le T/(8d)\le T/2\), we have $ Kn \ge T-K \ge T/2$, which further implies that $K\sqrt n = \sqrt{K(Kn)} \ge \sqrt{KT/2}$. Utilizing \(K\ge(1+\mathcal P_{T,1})/3\), we obtain
\begin{align*}
    \mathfrak R_T^\Delta(\mathcal P_{T,1}) \ge c_1G \sqrt{\frac{d}{\log d}} \sqrt{\frac{KT}{2}} \ge cG \sqrt{\frac{dT(1+\mathcal P_{T,1})}{\log d}},
\end{align*}
for some constant \(c>0\). The proof is completed.
\end{proof}
\begin{remark}
    We point out a subtle difference between the oracle models used in the upper and lower bounds. The upper bound in Theorem \ref{proposition.BMD simplex with entropy} is proved under an extension-oracle model, where the perturbed points may lie in $\mathcal{U}$ but outside the simplex. In contrast, the lower bound is stated for the standard in-domain simplex bandit model, where the learner queries two points in the simplex at each round. Closing the remaining logarithmic gaps and developing fully in-domain simplex perturbation schemes are interesting directions for future work.
\end{remark}

\section{Conclusion}
In this paper, we studied non-stationary bandit convex optimization with two-point feedback under the dynamic regret. We proposed a parameter-free bandit mirror descent method which enjoys a near-optimal regret. By exploiting the second-order moment of the gradient estimator, the proposed method attains the optimal $\cO(\sqrt{dT(1+\cP_{T,2})})$ dynamic regret in Euclidean geometry and extends to non-Euclidean domains. It further achieves $\cO(\sqrt{d\log(d)T(1+\cP_{T,p})})$ regret over the cross-polytope for $p=1+1/\log d$, and $O(\sqrt{d\log(d)T\log(T)(1+\cP_{T,1})})$ regret over the simplex.

\section*{Appendix: Technical Lemmas}
\begin{lemma}\label{lemma.cauchy schwarz with p norm}
    For any $\bx, \by \in \reals^d$, $p \in [1,\infty]$ and $\varepsilon > 0$, it holds that
    \begin{align*}
        \inner{\bx}{\by} \le \frac{\varepsilon}{2}\|\bx\|_p^2 + \frac{1}{2\varepsilon}\|\by\|_{p^*}^2,
    \end{align*}
    where $p^*$ is the conjugate of $p$, i.e., $1/p + 1/p^* = 1$.
\end{lemma}
\begin{proof}
    Note that $\|\cdot\|_{p^*}$ is the dual of $\|\cdot\|_{p}$. From the definition of dual norm, we have
    \begin{align*}
        \inner{\bx}{\by} \le \|\bx\|_p \|\by\|_{p^*}.
    \end{align*}
    Then we complete the proof by the fact that $ab \le \frac{\varepsilon}{2}a^2 + \frac{1}{2\varepsilon}b^2$ for $a, b \in \reals$ and $\varepsilon > 0$.
\end{proof}

\begin{lemma}\label{lemma.vector norm}
    For any vector $\bx \in \reals^d$, it holds that
    \begin{align*}
        \|\bx\|_q \le \|\bx\|_p \le d^{1/p - 1/q}\|\bx\|_q, \quad \forall \ 1 \le p \le q \le \infty.
    \end{align*}
    Consequently, for any $p, q \in [1,\infty]$, we have
    \begin{align*}
       \|\bx\|_q \le d^{1/q - 1/\max\{q,p\}} \|\bx\|_p.
    \end{align*}
\end{lemma}
\begin{proof}
    Fix $1 \le p \le q \le \infty$. The inequality $\|\bx\|_q \le \|\bx\|_p$ is the monotonicity of $\ell_r$-norms. For the reverse inequality, Hölder's inequality gives
    \[
        \|\bx\|_p^p = \sum_{i=1}^d |x_i|^p \cdot 1
        \le \left(\sum_{i=1}^d |x_i|^q\right)^{p/q}\left(\sum_{i=1}^d 1^{q/(q-p)}\right)^{1-p/q}
        = d^{1-p/q}\|\bx\|_q^p.
    \]
    Taking the $p$-th root yields $\|\bx\|_p \le d^{1/p - 1/q}\|\bx\|_q$. For the second claim, if $q \ge p$, we then have
    \[
        \|\bx\|_q \le \|\bx\|_p = d^{1/q - 1/\max\{q,p\}}\|\bx\|_p,
    \]
    since $\max\{q,p\}=q$. If $q < p$, then applying the first part with the roles of $p$ and $q$ exchanged yields
    \[
        \|\bx\|_q \le d^{1/q - 1/p}\|\bx\|_p = d^{1/q - 1/\max\{q,p\}}\|\bx\|_p.
    \]
    We complete the proof.
\end{proof}

\begin{lemma}[\citet{chen1993convergence}]\label{lemma.three point bregman}
    Let $B_\psi$ be the Bregman divergence with respect to $\psi: \operatorname{int}(\mathcal{X}) \rightarrow \mathbb{R}$. Then, for any three points $\bx, \by \in \operatorname{int}(\mathcal{X})$ and $\bz \in \mathcal{X}$, the following identity holds
    \begin{equation*}
        B_\psi(\bz ; \bx)+B_\psi(\bx ; \by)-B_\psi(\bz ; \by)=\langle\nabla \psi(\by)-\nabla \psi(\bx), \bz-\bx\rangle.
    \end{equation*}
\end{lemma}

\begin{lemma}\label{lemma.derivative of p norm}
    Let $p \in (1,\infty)$. For any $\bx \in \reals^d \setminus \{0\}$, the function $\bx \mapsto \|\bx\|_p$ is differentiable and
    \begin{align*}
        \frac{\partial}{\partial x_j} \|\bx\|_p = \frac{x_j |x_j|^{p-2}}{\|\bx\|_p^{p-1}}, \quad j = 1, \ldots, d.
    \end{align*}
\end{lemma}
\begin{proof}
    Since
    \[
        \|\bx\|_p = \left(\sum_{i=1}^d |x_i|^p\right)^{1/p}
    \]
    and $p>1$, the map is differentiable at every $\bx \neq 0$. By the chain rule, we have
    \begin{align*}
        \frac{\partial}{\partial x_j} \|\bx\|_p
        = \frac{1}{p}\left(\sum_{i=1}^d |x_i|^p\right)^{1/p-1} \cdot p |x_j|^{p-2}x_j
        = \frac{x_j |x_j|^{p-2}}{\|\bx\|_p^{p-1}}.
    \end{align*}
   The proof is completed.
\end{proof}

\bibliographystyle{abbrvnat}
\bibliography{reference}

@article{liu2025non,
  title={Non-stationary Bandit Convex Optimization: A Comprehensive Study},
  author={Liu, Xiaoqi and Baudry, Dorian and Zimmert, Julian and Rebeschini, Patrick and Akhavan, Arya},
  journal={arXiv preprint arXiv:2506.02980},
  year={2025}
}

@inproceedings{duchi2010composite,
  title={Composite objective mirror descent.},
  author={Duchi, John C and Shalev-Shwartz, Shai and Singer, Yoram and Tewari, Ambuj},
  booktitle={Colt},
  volume={10},
  pages={14--26},
  year={2010},
  organization={Citeseer}
}

@inproceedings{hazan2009efficient,
  title={Efficient learning algorithms for changing environments},
  author={Hazan, Elad and Seshadhri, Comandur},
  booktitle={Proceedings of the 26th annual international conference on machine learning},
  pages={393--400},
  year={2009}
}

@article{chen2018bandit,
  title={Bandit convex optimization for scalable and dynamic IoT management},
  author={Chen, Tianyi and Giannakis, Georgios B},
  journal={IEEE Internet of Things Journal},
  volume={6},
  number={1},
  pages={1276--1286},
  year={2018},
  publisher={IEEE}
}

@inproceedings{mokhtari2016online,
  title={Online optimization in dynamic environments: Improved regret rates for strongly convex problems},
  author={Mokhtari, Aryan and Shahrampour, Shahin and Jadbabaie, Ali and Ribeiro, Alejandro},
  booktitle={2016 IEEE 55th Conference on Decision and Control (CDC)},
  pages={7195--7201},
  year={2016},
  organization={IEEE}
}

@inproceedings{yang2016tracking,
  title={Tracking slowly moving clairvoyant: Optimal dynamic regret of online learning with true and noisy gradient},
  author={Yang, Tianbao and Zhang, Lijun and Jin, Rong and Yi, Jinfeng},
  booktitle={International Conference on Machine Learning},
  pages={449--457},
  year={2016},
  organization={PMLR}
}

@inproceedings{jadbabaie2015online,
  title={Online optimization: Competing with dynamic comparators},
  author={Jadbabaie, Ali and Rakhlin, Alexander and Shahrampour, Shahin and Sridharan, Karthik},
  booktitle={Artificial Intelligence and Statistics},
  pages={398--406},
  year={2015},
  organization={PMLR}
}

@article{besbes2015non,
  title={Non-stationary stochastic optimization},
  author={Besbes, Omar and Gur, Yonatan and Zeevi, Assaf},
  journal={Operations research},
  volume={63},
  number={5},
  pages={1227--1244},
  year={2015},
  publisher={INFORMS}
}

@article{ghadimi2013stochastic,
  title={Stochastic first-and zeroth-order methods for nonconvex stochastic programming},
  author={Ghadimi, Saeed and Lan, Guanghui},
  journal={SIAM journal on optimization},
  volume={23},
  number={4},
  pages={2341--2368},
  year={2013},
  publisher={SIAM}
}

@inproceedings{zinkevich2003online,
  title={Online convex programming and generalized infinitesimal gradient ascent},
  author={Zinkevich, Martin},
  booktitle={Proceedings of the 20th international conference on machine learning (icml-03)},
  pages={928--936},
  year={2003}
}

@article{chen1993convergence,
  title={Convergence analysis of a proximal-like minimization algorithm using Bregman functions},
  author={Chen, Gong and Teboulle, Marc},
  journal={SIAM Journal on Optimization},
  volume={3},
  number={3},
  pages={538--543},
  year={1993},
  publisher={SIAM}
}

@article{dekel2015bandit,
  title={Bandit smooth convex optimization: Improving the bias-variance tradeoff},
  author={Dekel, Ofer and Eldan, Ronen and Koren, Tomer},
  journal={Advances in Neural Information Processing Systems},
  volume={28},
  year={2015}
}

@inproceedings{saha2011improved,
  title={Improved regret guarantees for online smooth convex optimization with bandit feedback},
  author={Saha, Ankan and Tewari, Ambuj},
  booktitle={Proceedings of the fourteenth international conference on artificial intelligence and statistics},
  pages={636--642},
  year={2011},
  organization={JMLR Workshop and Conference Proceedings}
}

@inproceedings{flaxman2005online,
  title={Online convex optimization in the bandit setting: gradient descent without a gradient},
  author={Flaxman, Abraham D and Kalai, Adam Tauman and McMahan, H Brendan},
  booktitle={Proceedings of the sixteenth annual ACM-SIAM symposium on Discrete algorithms},
  pages={385--394},
  year={2005}
}

@article{shamir2017optimal,
  title={An optimal algorithm for bandit and zero-order convex optimization with two-point feedback},
  author={Shamir, Ohad},
  journal={Journal of Machine Learning Research},
  volume={18},
  number={52},
  pages={1--11},
  year={2017}
}

@article{duchi2015optimal,
  title={Optimal rates for zero-order convex optimization: The power of two function evaluations},
  author={Duchi, John C and Jordan, Michael I and Wainwright, Martin J and Wibisono, Andre},
  journal={IEEE Transactions on Information Theory},
  volume={61},
  number={5},
  pages={2788--2806},
  year={2015},
  publisher={IEEE}
}

@inproceedings{shao2024improved,
  title={Improved Dimensionality Dependence for Zeroth-Order Optimisation over Cross-Polytopes},
  author={Shao, Weijia},
  booktitle={Forty-first International Conference on Machine Learning},
  year={2024}
}

@book{cesa2006prediction,
  title={Prediction, learning, and games},
  author={Cesa-Bianchi, Nicolo and Lugosi, G{\'a}bor},
  year={2006},
  publisher={Cambridge university press}
}

@article{zhao2024adaptivity,
  title={Adaptivity and non-stationarity: Problem-dependent dynamic regret for online convex optimization},
  author={Zhao, Peng and Zhang, Yu-Jie and Zhang, Lijun and Zhou, Zhi-Hua},
  journal={Journal of Machine Learning Research},
  volume={25},
  number={98},
  pages={1--52},
  year={2024}
}

@article{zhang2018adaptive,
  title={Adaptive online learning in dynamic environments},
  author={Zhang, Lijun and Lu, Shiyin and Zhou, Zhi-Hua},
  journal={Advances in neural information processing systems},
  volume={31},
  year={2018}
}

@book{zhou2012ensemble,
  title={Ensemble methods: foundations and algorithms},
  author={Zhou, Zhi-Hua},
  year={2012},
  publisher={CRC press}
}

@article{zhao2021bandit,
  title={Bandit convex optimization in non-stationary environments},
  author={Zhao, Peng and Wang, Guanghui and Zhang, Lijun and Zhou, Zhi-Hua},
  journal={Journal of Machine Learning Research},
  volume={22},
  number={125},
  pages={1--45},
  year={2021}
}

@article{akhavan2022gradient,
  title={A gradient estimator via $\ell_1$-randomization for online zero-order optimization with two point feedback},
  author={Akhavan, Arya and Chzhen, Evgenii and Pontil, Massimiliano and Tsybakov, Alexandre},
  journal={Advances in Neural Information Processing Systems},
  volume={35},
  pages={7685--7696},
  year={2022}
}

@inproceedings{agarwal2010optimal,
  title={Optimal Algorithms for Online Convex Optimization with Multi-Point Bandit Feedback.},
  author={Agarwal, Alekh and Dekel, Ofer and Xiao, Lin},
  booktitle={Colt},
  pages={28--40},
  year={2010},
  organization={Citeseer}
}

@article{gao2018information,
  title={On the information-adaptive variants of the ADMM: an iteration complexity perspective},
  author={Gao, Xiang and Jiang, Bo and Zhang, Shuzhong},
  journal={Journal of Scientific Computing},
  volume={76},
  pages={327--363},
  year={2018},
  publisher={Springer}
}

\end{document}